\documentclass[11pt]{article}
\usepackage{latexsym}
\usepackage{amssymb}
\usepackage{amsmath}
\usepackage{amsfonts}
\usepackage{textcomp}
\usepackage{amscd}
\usepackage{amsthm}
\usepackage{mathrsfs}
\usepackage{amsxtra}
\usepackage{stmaryrd}
\usepackage[all]{xy}
\usepackage[dvips]{graphicx}

\newcommand{\BibTeX}{{\scshape Bib}\kern-.0em\TeX}
\newcommand{\T}{\S\kern .0em\relax }
\newcommand{\AMS}{$\mathcal{A}$\kern-.0em\lower.0ex\hbox
        {$\mathcal{M}$}\kern-.0em$\mathcal{S}$}

\usepackage{amsfonts}
\usepackage{color}
\usepackage{colordvi}
\usepackage{multicol}

\def\ng{\mathrm{ng}}

\def\irr{\mathrm{irr}}
\def\LA{\mathrm{LA}}

\def\D{\mathrm{D}}

\def\calL{\mathcal{L}}

\def\SA{\mathscr{A}}

\def\SS{\mathscr{S}}
\def\SL{\mathscr{L}}
\def\SC{\mathscr{C}}
\def\SD{\mathscr{D}}
\def\SO{\mathscr{O}}

\def\rig{\mathrm{rig}}

\def\an{\mathrm{an}}

     \newcommand{\BN}{{\mathbb {N}}}
     
    \newcommand{\BQ}{{\mathbb {Q}}} \newcommand{\BR}{{\mathbb {R}}}

     \newcommand{\BZ}{{\mathbb {Z}}}

     \newcommand{\RB}{{\mathrm {B}}}
     \newcommand{\RD}{{\mathrm {D}}}

     \newcommand{\RP}{{\mathrm {P}}}
     
     \newcommand{\RT}{{\mathrm {T}}}

     \newcommand{\RZ}{{\mathrm {Z}}}

    \newcommand{\Gal}{{\mathrm{Gal}}} \newcommand{\GL}{{\mathrm{GL}}}
    \newcommand{\Hom}{{\mathrm{Hom}}} 
    \newcommand{\Ind}{{\mathrm{Ind}}}

    \newcommand{\Res}{{\mathrm{Res}}}

\newcommand{\cris}{\mathrm{cris}}

\newcommand{\wvec}[4]{{\scriptsize{\big [ \!\!
    \begin{array}{cc} #1 \!\!\! & \!\!\! #2 \\ #3 \!\!\! & \!\!\! #4 \end{array} \!\! \big ] }}}

\def\res{\mathrm{res}}

\def\Rep{\mathrm{Rep}}

\def\LP{\mathrm{LP}}

\def\SE{\mathscr{E}}
\def\SR{\mathscr{R}}
\def\SO{\mathscr{O}}

\def\limproj{\mathop{\oalign{{\rm lim}\cr
\hidewidth$\longleftarrow$\hidewidth\cr}}}
     \newcommand{\Sym}{{\mathrm{Sym}}}

    \newcommand{\st}{{\mathrm{st}}}

    \newcommand{\ra}{\rightarrow}

    \theoremstyle{plain}
    \newtheorem{thm}{Theorem}[section] \newtheorem{cor}[thm]{Corollary}
    \newtheorem{lem}[thm]{Lemma}  \newtheorem{prop}[thm]{Proposition}
    \newtheorem {conj}[thm]{Conjecture}

    \theoremstyle{definition}

    \theoremstyle{remark}
    \newtheorem {rem}[thm]{Remark}
    \newtheorem {example}[thm]{Example}

    \numberwithin{equation}{section}

 \newcommand{\binc}[2]{ \bigg (\!\! \begin{array}{c} #1\\
    #2 \end{array}\!\! \bigg )}

\begin{document}
\title{Locally analytic vectors of unitary principal series of $\GL_2(\BQ_p)$ }

\author{Ruochuan Liu\\ University of Michigan, Ann Arbor\\ ruochuan@umich.edu\\  \\
Bingyong Xie\\ East China Normal University, Shanghai\\byxie@math.ecnu.edu.cn\\  \\
Yuancao Zhang\\ Peking University, Beijing\\zhangyc@math.pku.edu.cn}
\maketitle

\begin{abstract}
The $p$-adic local Langlands correspondence for
$\GL_2(\BQ_p)$ attaches to any $2$-dimensional irreducible $p$-adic representation $V$ of
$G_{\BQ_p}$ an admissible unitary
representation $\Pi(V)$ of $\GL_2(\BQ_p)$. The unitary principal series of $\GL_2(\BQ_p)$
are those $\Pi(V)$ corresponding to trianguline
representations. In this article, for $p>2$, using the machinery of Colmez,
we determine the space of locally analytic vectors $\Pi(V)_\an$ for all non-exceptional
unitary principal series $\Pi(V)$ of $\GL_2(\BQ_p)$
by proving a conjecture of Emerton.
\end{abstract}

\tableofcontents

\section{Introduction}

Let $F$ be a finite extension of $\BQ_p$. The aim of the $p$-adic local Langlands programme initiated by Breuil is to look for a ``natural'' correspondence between certain $n$-dimensional $p$-adic representations of $\mathrm{Gal}(\overline{\BQ}_p/F)$ and certain Banach space representations of $\GL_n(F)$. Thanks to much recent work, especially that of Colmez and Pa\v{s}k\={u}nas, we have gained a fairly clear picture in the case $F=\BQ_p$ and $n=2$ which is the so-called $p$-adic local Langlands correspondence for $\GL_2(\BQ_p)$ establishing a functorial bijection between $2$-dimensional
irreducible $p$-adic representations of $\mathrm{Gal}(\overline{\BQ}_p/\BQ_p)$ and non-ordinary irreducible admissible unitary representations of $\GL_2(\BQ_p)$.

Although the present version of $p$-adic local Langlands correspondence is formulated at the level of Banach space representations, it is very useful, as in Breuil's initial work \cite{Breuil-DS}, to have information of the subspace of locally analytic vectors. Fix a finite extension $L$ of $\BQ_p$ as the coefficient field, and we denote by $\Pi(V)$
the corresponding unitary representation of $\GL_2(\BQ_p)$ for any $2$-dimensional
irreducible $L$-linear representation $V$ of $\mathrm{Gal}(\overline{\BQ}_p/\BQ_p)$.
The \emph{unitary principal series of $\GL_2(\BQ_p)$}, which are the simplest ones among these $\Pi(V)$,
are those corresponding to trianguline representations. In \cite{Emerton-lg}, Emerton made a conjectural description of the subspace of locally analytic vectors $\Pi(V)_\an$ for all unitary principal series $\Pi(V)$. We recall his conjecture below.

Let $\SS_{\mathrm{irr}}$ be the parameterizing space of $2$-dimensional irreducible trianguline representations of $\mathrm{Gal}(\overline{\BQ}_p/\BQ_p)$ introduced by Colmez in \cite{Colmez08}. For any $s\in\SS_{\mathrm{irr}}$, let $V(s)$ be the corresponding trianguline representation. We may write $s=(\delta_1,\delta_2,\SL)$ so that the \'etale $(\varphi,\Gamma)$-module $\D_{\rig}(V(s))$ is isomorphic to the extension of $\SR(\delta_2)$ by $\SR(\delta_1)$ defined by $\SL$. For any such $s$, if $\delta_1\delta_2^{-1}=x^k|x|$ for some $k\in\BZ_{+}$, then we set $\Sigma(s)$ to be the locally analytic $\GL_2(\BQ_p)$-representation $\Sigma(k+1,\SL)\otimes((\delta_2|x|^{\frac{2-k}{2}})\circ\det)$ where $\{\Sigma(k+1,\SL)\}$ is the family of locally analytic $\GL_2(\BQ_p)$-representations introduced by Breuil in \cite{Breuil-DS}. Otherwise, we define $\Sigma(s)$ to be the locally analytic principal series
$\left(\Ind^{\GL_2(\BQ_p)}_{\mathrm{B}(\BQ_p)}\delta_2\otimes\delta_1(x|x|)^{-1}\right)
_\an$. The conjecture of Emerton is:

\begin{conj}$($\cite[Conjecture 6.7.3, 6.7.7]{Emerton-lg}$)$ \label{conj:Emerton-intro}
 For any $s\in\SS_{\irr}$, $\Pi(V(s))_\an$ sits in an
exact sequence
\begin{equation} \label{eq:Emerton-conj-intro}
 0 \longrightarrow \Sigma(s) \longrightarrow\Pi(V(s))_\an \longrightarrow \left(\Ind^{\GL_2(\BQ_p)}_{\mathrm{B}(\BQ_p)}\delta_1\otimes\delta_2(x|x|)^{-1}\right)_\an\longrightarrow 0.
\end{equation}
\end{conj}

In the special cases when $V(s)$ are twists of crystabelian representations and \emph{non-exceptional}, there is a more precise conjectural description of $\Pi(V(s))_\an$ due to Breuil. In \cite{Liu}, the first author proved Breuil's conjecture. The main result of this paper is:
\begin{thm}\label{thm:Emerton-intro} $(\mathrm{Theorem}\ \ref{thm:emerton})$
For $p>2$, Conjecture \ref{conj:Emerton-intro} is true if $V(s)$ is non-exceptional.
\end{thm}
\noindent
In fact, one can easily deduce Breuil's conjecture from Emerton's conjecture. Thus for $p>2$, the above theorem covers the aforementioned result of the first author.

We now explain the strategy of the proof of Theorem \ref{thm:Emerton-intro}. The whole proof builds on Colmez's machinery of $p$-adic local Langlands correspondence for $\GL_2(\BQ_p)$ developed in \cite{Colmez-Langlands}. The key ingredient is Colmez's identification of the locally analytic vectors:
\begin{equation}\label{eq:formula-analytic}
(\Pi(\check{V})_\an)^*=D_\rig^\natural\boxtimes\RP^1
\end{equation}
where $D=\D(V)$ is Fontaine's \'etale $(\varphi,\Gamma)$-module associated to $V$. To apply this formula, for any continuous characters $\delta, \eta:\BQ_p^\times\ra L^\times$, we construct the objects $\SR(\eta)\boxtimes_\delta\RP^1$ and $\SR^+(\eta)\boxtimes_\delta\RP^1$ which are equipped with continuous $\GL_2(\BQ_p)$-actions, and the latter is topologically isomorphic to $(\left(\Ind^{\GL_2(\BQ_p)}_{\mathrm{B}(\BQ_p)}\delta^{-1}\eta\otimes\eta^{-1}\right)
_\an)^*$. In doing so, we are led to modify some of Colmez's constructions to twists of \'etale $(\varphi,\Gamma)$-modules and rank 1 $(\varphi,\Gamma)$-modules. On the other hand, Colmez \cite{Colmez-principal} (for $p>2$ and $s\in \SS_*^{\mathrm{ng}}\coprod\SS_{*}^{\mathrm{st}}$; this is the only place where we need $p>2$) and Berger-Breuil \cite{BB} (for $s\in\SS_\irr$ non-exceptional) establish an explicit isomorphism $\SA_s:\Pi(V(s))\cong\Pi(s)$ (for $s$ exceptional, Paskunas proves $\Pi(V(s))\cong\Pi(s)$ by an indirect method \cite{P09}) where $\Pi(s)$ is the unitary principal series associated to $V(s)$.  We deduce from the explicit description of $\SA_s$ plus a duality argument the following exact sequence

\begin{equation}\label{eq:exact-intro}
0\longrightarrow\SR(\delta_1) \boxtimes\RP^1 \longrightarrow
\D_\rig(V(s))\boxtimes \RP^1 \longrightarrow
\SR(\delta_2)\boxtimes\RP^1\longrightarrow 0.
\end{equation}
Then the natural inclusion $\left(\Ind^{\GL_2(\BQ_p)}_{\mathrm{B}(\BQ_p)}\delta_2\otimes\delta_1(x|x|)^{-1}\right)
_\an\hookrightarrow\Pi(s)_\an$ induces the following commutative diagram
\begin{equation} \label{eq:comm-diag-intro}
\xymatrix{
(\Pi(s)_\an)^* \ar[r] \ar[d] &
D^\natural_\rig(\check{s})\boxtimes \RP^1 \ar[d]\\
(\left(\Ind^{\GL_2(\BQ_p)}_{\mathrm{B}(\BQ_p)}\delta_2\otimes\delta_1(x|x|)^{-1}\right)
_\an)^*
\ar[r] & \SR(\check{\delta}_1)\boxtimes\RP^1}
\end{equation}
where $\check{s}=(\check{\delta}_2,\check{\delta}_1,\SL)$ corresponds to the Tate dual of $V(s)$. Using (\ref{eq:comm-diag-intro}) together with the fact that $D^\natural_\rig(s)\boxtimes \RP^1$ and $D^\natural_\rig(\check{s})\boxtimes \RP^1$ are orthogonal complements of each other under the paring $$\{\cdot,\cdot\}_{\RP^1}:\D_\rig(V(s))\boxtimes \RP^1\times\D_\rig(V(\check{s}))\boxtimes \RP^1\ra L,$$
and that (\ref{eq:exact-intro}) is dual to
\begin{equation}\label{eq:dual-exact-intro}
0\longrightarrow\SR(\check{\delta}_2) \boxtimes\RP^1 \longrightarrow
\D_\rig(V(\check{s}))\boxtimes \RP^1\longrightarrow
\SR(\check{\delta}_1)\boxtimes\RP^1\longrightarrow 0,
\end{equation}
we deduce that if $\delta_1\delta_2^{-1}\neq x^k|x|$ for any $k\in\BZ_+$, then $D^\natural_\rig(\check{s})$ sits in the exact sequence
\begin{equation}\label{eq:analytic-exact-intro}
0\longrightarrow\SR^+(\check{\delta}_2) \boxtimes\RP^1 \longrightarrow
D^\natural_\rig(\check{s})\boxtimes \RP^1 \longrightarrow
\SR^+(\check{\delta}_1)\boxtimes\RP^1\longrightarrow 0.
\end{equation}
We therefore conclude (\ref{eq:Emerton-conj-intro}) by taking the dual of (\ref{eq:analytic-exact-intro}). If $\delta_1\delta_2^{-1}=x^k|x|$ for some $k\in\BZ_+$, we have that the image of $D^\natural_\rig(\check{s})$ in $\SR(\check{\delta}_1)\boxtimes\RP^1$ is a closed subspace of $\SR^+(\check{\delta}_1)\boxtimes\RP^1$ of codimension $k$ and $\SR(\check{\delta}_2) \boxtimes\RP^1\cap
D^\natural_\rig(\check{s})\boxtimes \RP^1$ contains $\SR^+(\check{\delta}_2) \boxtimes\RP^1$ as a closed subspace of codimension $k$. We then conclude (\ref{eq:Emerton-conj-intro}) using Schneider and Teitelbaum's results on the Jordan-H\"{o}lder series of locally analytic principal series of $\GL_2(\BQ_p)$.

The organization of the paper is as follows. In \S2, we recall some necessary background of the theory of $(\varphi,\Gamma)$-modules. In \S3, we recall some of Colmez's constructions of the $p$-adic local Langlands correspondence for $\GL_2(\BQ_p)$ especially his identification of the locally analytic vectors, and we make the aforementioned modification. We review the isomorphism $\SA_s:\Pi(V(s))\cong\Pi(s)$ in \S4. In \S5.1, we recall Schneider and Teitelbaum's results on Jordan-H\"{o}lder series of the locally analytic principal series of $\GL_2(\BQ_p)$. We prove that  $\SR^+(\eta)\boxtimes_\delta\RP^1$ is isomorphic to $(\left(\Ind^{\GL_2(\BQ_p)}_{\mathrm{B}(\BQ_p)}\delta^{-1}\eta\otimes\eta^{-1}\right)
_\an)^*$ in \S5.2.  Section 6 is devoted to the proof of Theorem \ref{thm:Emerton-intro}. We first recall the definition of $\Sigma(s)$ and restate Emerton's conjecture in  \S6.1. Then we prove the exact sequence (\ref{eq:exact-intro}) in \S6.2. We finish the proof of Theorem \ref{thm:Emerton-intro} in \S6.3.

After the work presented in this paper was finished, we learned from Colmez that he had a proof of Conjecture \ref{conj:Emerton-intro} for all $p$ and all trianguline representations of $G_{\BQ_p}$. The strategy of his proof is different from ours. He constructs the map $\Pi(s)_\an\ra\left(\Ind^{\GL_2(\BQ_p)}_{\mathrm{B}(\BQ_p)}\delta_1\otimes\delta_2(x|x|)^{-1}\right)_\an$ directly by computing \emph{module de Jacquet dual} of $\Pi(s)_\an$. We refer the reader to his paper \cite{Colmez-analytic} for more details.

\section*{Notation and conventions}
Let $p$ be a prime number, and let $v_p$ denote the $p$-adic valuation on $\overline{\mathbb{Q}}_p$, normalized by $v_p(p)=1$; the corresponding norm is denoted by $|\cdot|$. Let $G_{\BQ_p}$ denote $\mathrm{Gal}(\overline{\BQ}_p/\BQ_p)$ for simplicity. Let $\chi:G_{\BQ_p}\ra\BZ_p^\times$ be the $p$-adic cyclotomic character. The kernel of $\chi$ is $H=\Gal(\overline{\BQ}_p/\BQ_p(\mu_{p^\infty}))$, and let $\Gamma=\Gal(\BQ_p(\mu_{p^\infty})/\BQ_p)$. For
any $a\in\BZ_p^\times$, let $\sigma_a$ be the unique element in $\Gamma$ such
that $\chi(\sigma_a)=a$. For any positive integer $h$, let
$\Gamma_h=\chi^{-1}(1+p^h\BZ_p)$. If we regard $\chi$ as a character of $\BQ_p^\times$ via the local Artin map, then it is equal to
$\epsilon(x)=x|x|$. Throughout this paper, we fix a finite extension $L$ of $\BQ_p$. We denote by $\mathcal{O}_L$ the ring of integers of $L$, and by $k_L$ the residue field. Let $\widehat{\mathscr{T}}(L)$ be the set of all continuous characters $\delta: G_{\BQ_p}\ra L^\times$. For any $\delta\in \widehat{\mathscr{T}}(L)$,
the Hodge-Tate weight $w(\delta)$ of $\delta$ is defined by $w(\delta)=\frac{\log \delta(u)}{\log u}$ where
$u$ is any element of $\BQ_p^\times\setminus\mu_{p^\infty}$. For any $L$-linear representation $V$ of $G_{\BQ_p}$, we denote by $\check{V}$ the Tate dual $V^*(\epsilon)$ of $V$. For any $\SL\in L$, let $\log_\SL: \BQ_p^\times
\rightarrow L$ be the homomorphism defined by $\log_\SL(p)$=1 and
$\log_\SL(x)=-\sum\limits_{n=1}^{+\infty} \frac{(1-x)^{n}}{n}$ when
$|x-1|<1$. We put $\log_\infty=v_p$. Hence $\log_\SL$ is defined for all $\SL\in\RP^1(L)$.

Let $\RB$ be the subgroup of upper triangular matrices of $\GL_2$, let $\RP=\wvec{*}{*}{0}{1}$ be the mirabolic subgroup of $\GL_2$, let $\RT$ be the subgroup of diagonal matrices of $\GL_2$, and let $\RZ$ be the center of $\GL_2$. Put $w=\wvec{0}{1}{1}{0}$. Let $\mathrm{Rep}_{\mathrm{tors}}\GL_2(\BQ_p)$ be the category of smooth $\mathcal{O}_L[\GL_2(\BQ_p)]$-modules which are of finite lengths and admit central characters. Let $\mathrm{Rep}_{\mathcal{O}_L}\GL_2(\BQ_p)$ be the category of $\mathcal{O}_L[\GL_2(\BQ_p)]$-modules $\Pi$ which are separated and complete for the $p$-adic topology, $p$-torsion free, and $\Pi/p^n\Pi\in\mathrm{Rep}_{\mathrm{tors}}\GL_2(\BQ_p)$ for any $n\in\mathbb{N}$.

\section*{Acknowledgements}
The first author thanks Christophe Breuil and Pierre Colmez for helpful discussions and communications, and Liang Xiao for helpful comments on earlier drafts of this paper.
The second and the third authors thank
Qingchun Tian for helpful discussions and encouragements. The first author would also like to thank Henri Darmon for his various help during the first author's stay at McGill university.
The first author wrote this paper as a postdoctoral fellow of Centre de Recherches Math\'ematiques. Part of this work were done while the first author was a visitor at
Institut des Hautes \'Etudes Scientifiques and Beijing International Center for
Mathematical Research. The first author is grateful to these institutions for their hospitality. While writing this paper, the second author was supported by
the postdoctoral grant 533149-087 of Peking University and Beijing International Center for Mathematical Research.

\section{Preliminaries on $(\varphi,\Gamma)$-modules}

\subsection{Dictionary of $p$-adic functional analysis}
Let $\mathcal{O}_{\SE}$ be
the ring of Laurent series $f=\sum_{i\in\BZ} a_i T^i$, where $a_i\in \mathcal{O}_L$, such that $v_p(a_i)\ra \infty$ as $i\ra-\infty$. Let $\SE=\mathcal{O}_{\SE}[1/p]$ be the fraction field of $\mathcal{O}_{\SE}$.

For any $r\in\BR_+\cup\{+\infty\}$, let $\SE^{]0,r]}$ be the ring of Laurent series $f=\sum_{i\in\BZ} a_i T^i$, with $a_i\in L$, such that $f$ is convergent on the annulus $0<v_p(T)\leq r$. For any $0<s\leq r\leq+\infty$, we define the valuation $v^{\{s\}}$ on $\SE^{]0,r]}$ by
$$v^{\{s\}}(f)=\inf_{i\in\BZ} \{v_p(a_i)+is\}\hspace{1mm}\text{if}\hspace{1mm}s\neq\infty;\hspace{1mm}v^{\{\infty\}}(f)=v_p(f(0)).$$
We equip $\SE^{]0,r]}$ with the Fr\'echet topology defined by the family of valuations $\{v^{\{s\}}|0< s\leq r\}$; then $\SE^{]0,r]}$ is complete. We equip the \emph{Robba ring} $\SR=\bigcup_{r>0}\SE^{]0,r]}$ with the inductive limit topology. We denote $\SE^{]0,+\infty]}$, the ring of analytic functions on the open unit disk, by $\SR^+$.

Let $\SE^\dagger$ be the subring of overconvergent elements of $\SE$, i.e.
$\SE^\dagger$ is the set of $f\in \SE$ such that $f(T)$ is convergent over some annulus $0<v_p(T)\leq r$.
Let $\SE^{(0,r]}=\SE^\dagger\cap\SE^{]0,r]}$. We equip $\SE^\dagger=\bigcup_{r>0}\SE^{(0,r]}$ with the inductive limit topology. We denote $\SE^{(0,\infty]}=\mathcal{O}_L[[T]][1/p]$ by $\SE^+$, and let $\mathcal{O}_{\SE^+}=\mathcal{O}_L[[T]]$.

Let $R$ denote any of $\mathcal{O}_{\SE^+}, \SE^+, \mathcal{O}_\SE,\SE, \SE^\dagger,\SR^+$ and $\SR$. We equip the ring $R$ with commuting continuous actions of $\varphi$ and $\Gamma$ defined by
\begin{equation}\label{eq:phi-gamma-action}
\varphi(f(T))=f((1+T)^p-1),\qquad\gamma(f(T))=f((1+T)^{\chi(\gamma)}-1), \quad \gamma\in\Gamma.
\end{equation}
If we view $R$ as a $\varphi(R)$-module, then $R$ is freely generated by $\{(1+T)^i|i=0,\dots,p-1\}$. Thus for any $y\in R$, we may write $y=\sum\limits_{i=0}^{p-1}(1+T)^i \varphi(y_i)$ for some uniquely determined $y_0,\dots,y_{p-1}\in R$, and we define the operator $\psi: R\rightarrow R$ by setting $\psi(y)=y_0$. It follows directly from the definition that $\psi$ is a left inverse to $\varphi$, and that $\psi$ commutes with the $\Gamma$-action. For any $f=\sum_{i\in\BZ}a_i T^i\in\SR$, we define the residue of the $1$-form
$\omega=f\cdot dT$ as $\mathrm{res}(\omega)=a_{-1};$ and for any $f\in\SR$, we define $\res_{0} (f)=\mathrm{res}(f\frac{dT}{1+T})$.

We denote by $\SC^0(\BZ_p,L)$ the space of continuous functions on $\BZ_p$
with values in $L$; this is an $L$-Banach space equipped with the supremum norm. Let $\LA(\BZ_p,L)$ denote the space of locally analytic functions on $\BZ_p$
with values in $L$. The classical results of Mahler and Amice
assert that the set of functions $\{\binom{x}{n}\}_{n\in\BN}$ constitutes an orthogonal basis of $\SC^0(\BZ_p,L)$, and that
for $f=\sum_{n\in\BN}a_n(f)\binom{x}{n}\in\SC^0(\BZ_p,L)$,  $f\in\LA(\BZ_p,L)$
if and only if there exists some $r>0$ such that $v_p(a_n(f))-rn\ra +\infty$ as $n\ra+\infty$.

For any $u\geq0$, we denote by $\SC^u(\BZ_p,L)$ the space of all $\SC^u$-functions on $\BZ_p$; this an $L$-Banach space (see \cite{Colmez-function} for more details).
We have $\LA(\BZ_p,L)\subset \SC^u(\BZ_p,L)\subseteq\SC^0(\BZ_p,L)$, and that $\LA(\BZ_p,L)$ is dense in $\SC^u(\BZ_p,L)$ for any $u\geq0$.

We denote by $\SD(\BZ_p,L), \SD_u(\BZ_p,L)$ the topological dual of $\LA(\BZ_p,L),\SC^u(\BZ_p,L)$ respectively. Note that the natural map $\SD_u(\BZ_p,L)\ra\SD(\BZ_p,L)$ is injective since
$\LA(\BZ_p,L)$ is dense in $\SC^u(\BZ_p,L)$. The elements of $\SD(\BZ_p,L)$ are called
{\it distributions on $\BZ_p$}. A distribution $\mu$ is called of  {\it order $u$} if $\mu\in\SD_u(\BZ_p,L)$.
We define the actions of $\varphi,\psi $ and $\Gamma$ on $\SD(\BZ_p,L)$
by the formulas
\begin{eqnarray*} \int_{\BZ_p} f \varphi(\mu) = \int_{\BZ_p} f(px)
\mu, \hskip 5pt \int_{\BZ_p} f\psi(\mu) = \int_{p\BZ_p} f
(p^{-1}x)\mu, \hskip 5pt  \int_{\BZ_p} f \sigma_a(\mu) =
\int_{\BZ_p} f(ax) \mu
\end{eqnarray*}
for any $f\in\LA(\BZ_p,L),\mu\in\SD(\BZ_p,L)$ and $a\in\BZ_p^\times$.

The {\it Amice transformation} $\SA$ on $\SD(\BZ_p,L)$ is defined by
$$\SA:\SD(\BZ_p,L)\ra L[[T]],\quad \SA(\mu)=\int_{\BZ_p}(1+T)^x\mu(x)=\sum_{n=0}^{+\infty} T^n \int_{\BZ_p} \binc{x}{n} \mu.
$$
It is an immediate consequence of the results of Mahler and Amice that the Amice transformation $\mu\mapsto \SA(\mu)$ induces topological isomorphisms
from $\SD_0(\BZ_p,L)$ and $\SD(\BZ_p,L)$ to $ \SE^+$ and $ \SR^+$ respectively which are compatible with the actions of $\varphi,\psi$ and $\Gamma$.

We denote by $\LA_c(\BQ_p,L)$ the space of compactly supported
$L$-valued locally analytic functions on $\BQ_p$, and denote by $\SD(\BQ_p,L)$ the topological dual of
$\LA_c(\BQ_p,L)$. The elements of $\SD(\BQ_p,L)$ are called \emph{distributions on $\BQ_p$}. For any $\mu\in\SD(\BQ_p,L)$, let
$\mu^{(n)}$ be the distribution on $\BZ_p$ defined by
\[
\int_{\BZ_p}f\mu^{(n)}=\int_{p^{-n}\BZ_p}f(p^n x)\mu
\]
for any $f\in
\LA(\BZ_p,L)$.  It follows that $\psi(\mu^{(n+1)})=\mu^{(n)}$, and that any sequence of
distributions $(\mu^{(n)})_{n\in \BN}$ on $\BZ_p$ so that
$\psi(\mu^{(n+1)})=\mu^{(n)}$ uniquely determines
a distribution $\mu$ on $\BQ_p$. The {\it Amice transformation $\SA(\mu)$} for $\mu\in\SD(\BQ_p,L)$ is then defined to be the sequence $(\SA(\mu^{(n)}))_{n\in\BN}$.

A distribution $\mu$ on $\BQ_p$ is said to be {\it of order $u$} if
 all $\mu^{(n)}$ are of order $u$. The distribution $\mu$ is said to be {\it globally of
order $u$}, if there is a constant $C_u(\mu)$ such that
$v_{\SD_u}(\mu^{(n)})\geq nu+C_u(\mu)$ for all $n\in \BN$. Let
$\SD_u(\BQ_p,L)$ denote the space of distributions on $\BQ_p$ globally
of order $u$.

\subsection{The category of $(\varphi,\Gamma)$-modules}
Keep notations as in \S2.1. We define a \emph{$(\varphi,\Gamma)$-module} over $R$ to be a finite free $R$-module $D$
equipped with commuting continuous semilinear $\varphi,\Gamma$-actions. When $R=\mathcal{O}_\SE$, the $(\varphi,\Gamma)$-module $D$ is called {\it \'etale} if $\varphi(D)$ generates $D$ as an $\mathcal{O}_\SE$-module.  When $R=\SE$, the $(\varphi,\Gamma)$-module $D$ is
called {\it \'etale} if it arises by base change from an \'etale $(\varphi,\Gamma)$-module over $\mathcal{O}_{\SE}$. A $(\varphi,\Gamma)$-module $D$ over
$\SE^\dagger$ is called {\it \'etale} if
$D\otimes_{\SE^\dagger}\SE$ is \'etale as an $(\varphi,\Gamma)$-module over
$\SE$. A $(\varphi,\Gamma)$-module $D$ over $\SR$ is called
{\it \'etale} if the underlying
$\varphi$-module is \emph{pure of slope $0$} in the sense of Kedlaya's slope theory \cite{Kedlaya08}.

\begin{example}
For any $\delta\in\widehat{\mathscr{T}}(L)$,
we define $R(\delta)$ to be the rank 1 $(\varphi,\Gamma)$-module over $R$ which has an $R$-base $e$ satisfying
\begin{equation}\label{eq:rank1}
\varphi(e)=\delta(p)e,\qquad \sigma_a(e)=\delta(a)e,\quad a\in\BZ_p^\times.
\end{equation}
Such an element $e$, which is unique up to a nonzero scalar (this is because $R^{\varphi=1,\Gamma=1}=L$ or $\mathcal{O}_L$), is called a \emph{standard basis} of $R(\delta)$.
\end{example}

Let $V$ be a $d$-dimensional $L$-linear representation of $G_{\BQ_p}$, and let $T$ be a $G_{\BQ_p}$-invariant $\mathcal{O}_L$-lattice of $V$. Let $\widehat{\SE^{\mathrm{ur}}}$ be the $p$-adic completion of the maximal unramified extension
of $\SE$, and let $\widehat{\mathcal{O}_\SE^{\mathrm{ur}}}$ be the ring of integers of $\widehat{\SE^{\mathrm{ur}}}$. The $\varphi,\Gamma$-actions on $\SE$ naturally extend to continuous actions, which we again denote by $\varphi,\Gamma$ respectively, on $\widehat{\SE^{\mathrm{ur}}}$. We define
\[
\RD(T)=(T\otimes_{\mathcal{O}_\SE}\widehat{\mathcal{O}_\SE^{\mathrm{ur}}})^{H}
\qquad (\text{resp}. \hspace{1mm} \RD(V)=(V\otimes_{\SE}\widehat{\SE^{\mathrm{ur}}})^{H}),
\]
which is a $(\varphi,\Gamma)$-module over $\mathcal{O}_{\SE}$ (resp. $\SE$). We define $\RD^\dagger(V)$ to be the maximal finite dimensional $\varphi,\Gamma$-stable $\SE^\dagger$-subspace of $\RD(V)$, and we define $ \RD_\rig(V) = \RD^\dagger(V)\otimes_{\SE^\dagger}
\SR$; then $\RD^\dagger(V)$ and $\RD_\rig(V)$ are $(\varphi,\Gamma)$-modules over $\SE^\dagger$ and $\SR$ respectively.

\begin{thm} $($Fontaine \cite{Fontaine90}, Cherbonnier-Colmez \cite{CC98}, Berger \cite{Berger}, \cite{Berger06}$)$  $\RD(T)$ $($resp. $\RD(V)$, $\RD^\dagger(V)$, $\RD_\rig(V)$$)$ is an \'etale $(\varphi,\Gamma)$-module of rank $d$. Furthermore,
the functor $T\mapsto \RD(T)$ $($resp. $V\mapsto\RD(V)$, $V\mapsto\RD^\dagger(V)$, $V\mapsto
\D_\rig(V)$$)$ is a rank preserving equivalence of
categories from the category of $\mathcal{O}_L$-linear $($resp.
$L$-linear$)$ $G_{\BQ_p}$-representations to the category of \'etale $(\varphi,\Gamma)$-modules
over $\mathcal{O}_\SE$ $($resp. $\SE$, $\SE^\dagger$, $\SR$$)$.
\end{thm}

Let $D$ be a $(\varphi,\Gamma)$-module over $R$. If $D$ is isomorphic to its $\varphi$-pullback, then for any
$y\in D$, we may write $y=\sum_{i=0}^{p-1} (1+T)^i \varphi(y_i)$
for some uniquely determined $y_i\in D$. We define $\psi: D\rightarrow D$ by setting $\psi(y)=y_0$. It follows that $\psi$ commutes with $\Gamma$ and satisfies
$$\psi(a\varphi(x))=\psi(a)x, \hskip 10pt \psi(\varphi(a)x)=a\psi(x)$$
for any $a\in R, \ x\in D$. In particular, $\psi$ is a left inverse to $\varphi$. Set $\Res_{p\BZ_p}(y)=\varphi\psi(y)$,
$\Res_{\BZ_p^\times}(y)=(1-\varphi\psi)(y)$, and denote $\Res_{p\BZ_p}(D)$, $\Res_{\BZ_p^\times}(D)$ by $D\boxtimes p\BZ_p$, $D\boxtimes\BZ_p^\times$ respectively.

For an \'etale $(\varphi,\Gamma)$-module $D$ over $\SO_\SE$, A {\it trellis} of $D$ is a compact
$\mathcal{O}_{\SE^+}$-submodule $N$ which $\SO_\SE$-linearly generates $D$. Colmez \cite{Treillis} proves that the set of $\psi$-stable treillis admits a unique minimal element $D^\natural$,and that $\psi$ is surjective on $D^\natural$. It follows from the uniqueness that $D^\natural$ is stable under the $\Gamma$-action. In the simplest case when $D=\mathcal{O}_\SE$, we have $D^\natural=\mathcal{O}_{\SE^+}$.
For an \'etale $(\varphi,\Gamma)$-module $D$ over $\SE$, if $D$ is the base change of an \'etale $(\varphi,\Gamma)$-module $D_0$ over $\mathcal{O}_\SE$, we define $D^\natural=D_0^\natural[1/p]$ which is independent of the choice of $D_0$.

We define the $\varphi,\Gamma$-actions on the rank 1 $R$-module $R\frac{dT}{1+T}$
formally by
\[
\varphi(x\frac{dT}{1+T})=\varphi(x)\frac{dT}{1+T}, \qquad
\gamma(x\frac{dT}{1+T})=\chi(\gamma)\gamma(x)\frac{dT}{1+T},\ x\in R.
\]
Then the rank 1 $(\varphi,\Gamma)$-module $R\frac{d T}{1+T}$ is isomorphic to $R(\epsilon)$. For any \'etale $(\varphi,\Gamma)$-module $D$ over $R$, the \'etale $(\varphi,\Gamma)$-module $\check{D}=\Hom_R(D,
R\frac{dT}{1+T})$ is called the \emph{Tate dual} of $D$. We define the pairing $\{ \cdot ,\cdot \}:
\check{D}\times D\rightarrow L$ by setting $\{x, y\}=\res_0(\sigma_{-1}(x)(y)).$
It follows that $\{ \cdot ,\cdot \}$ is perfect and satisfies
\begin{equation}\label{eq:pairing-formula}
\{ x, \varphi(y) \}=\{ \psi(x), y \}.
\end{equation}


\section{$p$-adic local Langlands correspondence for $\GL_2(\BQ_p)$}
\label{sec:colmez}


\subsection{Operator $w_{\delta}$}
For an \'etale $(\varphi,\Gamma)$-module $D$ over $\mathcal{O}_{\SE}$ or $\SE$ and a continuous character $\delta:\BQ_p^\times\ra \mathcal{O}_L^\times$, Colmez constructs the involution
$w_\delta:D\boxtimes\BZ_p^\times\rightarrow D\boxtimes\BZ_p^\times$ defined by
\begin{equation} \label{eq:wD}
\begin{split}
w_\delta(z)=\lim_{n\rightarrow +\infty} \sum_{i\in \BZ_p^\times \ \mathrm{mod} \
p^n} \delta(i^{-1}) (1+T)^i \sigma_{-i^{2}} \cdot \varphi^n\psi^n
((1+T)^{-i^{-1}}z).
\end{split}
\end{equation}
Note that the right hand side of (\ref{eq:wD}) only involves $\delta|_{\BZ_p^\times}$. Since $\delta(\BZ_p^\times)\subseteq \mathcal{O}_L^\times$ for any $\delta\in\widehat{\mathscr{T}}(L)$, (\ref{eq:wD}) is still convergent for any  $D$ which is a twist of an \'etale $(\varphi,\Gamma)$ one and $\delta\in\widehat{\mathscr{T}}(L)$. From now on, we suppose $D$ is a twist of an \'etale $(\varphi,\Gamma)$-module over $\mathcal{O}_{\SE}$ or $\SE$ and $\delta\in\widehat{\mathscr{T}}(L)$, and define $w_\delta:D\boxtimes\BZ_p^\times\rightarrow D\boxtimes\BZ_p^\times$ by (\ref{eq:wD}).  Let $D^\dagger, D_\rig$ denote the $(\varphi,\Gamma)$-modules over $\SE^\dagger,\SR$ corresponding to $D$.

\begin{example}\label{example:calE+-stable}
It follows by $(\ref{eq:wD})$ that $\mathcal{O}_{\SE^+}\boxtimes\BZ_p^\times\subset\mathcal{O}_{\SE}\boxtimes\BZ_p^\times$ is stable under $w_\delta$. By \cite[V]{Treillis}, one can describe the $w_\delta$-action on $(\mathcal{O}_{\SE^+})^{\psi=0}$ more explicitly. Namely, for any $f\in\mathscr{C}^0(\BZ_p, L)$ and $z\in\mathcal{O}_{\SE}^+\boxtimes\BZ_p^\times$,
\begin{equation}\label{eq:w_delta}
\int_{\BZ_p^\times}f(x)\mathscr{A}^{-1}(w_\delta(z))=\int_{\BZ_p^\times}\delta(x)f(1/x)\mathscr{A}^{-1}(z)
\end{equation}

\end{example}

For any abelian profinite group $C$, we denote by $\Lambda_C$ the complete group algebra $$\mathcal{O}_L[[C]]=\limproj\mathcal{O}_L[C/C']$$
where $C'$ goes through all open subgroups of $C$.
If $C$ is pro-$p$ cyclic, and if $c$ is a topological generator of $C$,
then $\Lambda_C$ is canonically isomorphic to the ring consisting of $g(c-1)$ for all $g(T)\in\mathcal{O}_L[[T]]$,
and we further define $R(C)$ to be the ring consisting of $g(c-1)$ for all $g(T)\in R$ for any $R$ of $\SE^+$, $\SE^\dagger$, $\mathcal{O}_{\SE}^\dagger$, $\SE$, $\mathcal{O}_{\SE}$, $\SR^+$
and $\SR$; this is independent of the choice of $c$.
Now we choose $d\geq1$ such that $\Gamma_d$ is pro-$p$ cyclic (in fact, we can choose $d=1$ if $p$ is odd, and $d=2$ if $p=2$), then we define $R(\Gamma)=\Lambda_{\Gamma}\otimes_{\Lambda_{\Gamma_d}} R(\Gamma_d)$ which is independent of the choice of $d$.
The topological rings $\mathcal{O}_\SE(\Gamma)$ or $\SE(\Gamma)$, $\SE^\dagger(\Gamma)$ and $\SR(\Gamma)$
naturally act on $D\boxtimes\BZ_p^\times$, $D^\dagger\boxtimes\BZ_p^\times$ and $D_\rig\boxtimes\BZ_p^\times$ respectively.
The following Lemma follows from the proof of \cite[Lemme V.2.2]{Colmez-Langlands}.
\begin{lem}\label{lem:anti}
For any $\gamma\in\Gamma$ and $z\in D\boxtimes\BZ_p^\times$, $w_\delta(\gamma(z))=\delta(\chi(\gamma))\gamma^{-1}(w_\delta(z))$.
\end{lem}

Let $\iota_\delta: R(\Gamma)\ra R(\Gamma)$ denote
the involution defined by $\gamma\rightarrow
\delta(\chi(\gamma))\gamma^{-1}$. It is an immediate consequence of Lemma \ref{lem:anti} that
the action of $w_\delta$ on
$D\boxtimes \BZ_p^\times$ is $\mathcal{O}_\SE(\Gamma)$-semilinear with respect to $\iota_{\delta}$, i.e.
\begin{equation} \label{eq:anti2}
w_\delta(\lambda(z)) =\iota_{\delta}(\lambda)(w_\delta(z)),
\quad\lambda \in \mathcal{O}_\SE(\Gamma), \ z\in
D\boxtimes \BZ_p^\times.
\end{equation}

Let $\eta\in\widehat{\mathscr{T}}(L)$.
\begin{prop}\label{prop:rank1}
$\SE^+(\eta)\boxtimes\BZ_p^\times$ is a free $\SE^+(\Gamma)$-module of rank 1. Furthermore, we have
\[
\SR^+(\eta)\boxtimes\BZ_p^\times=\SR^+(\Gamma)\otimes_{\SE^+(\Gamma)}\SE^+(\eta)\boxtimes\BZ_p^\times,\quad \SE^\dagger(\eta)\boxtimes\BZ_p^\times=\SE^\dagger(\Gamma)\otimes_{\SE^+(\Gamma)}\SE^+(\eta)\boxtimes\BZ_p^\times.
\]
 As a consequence,
$\SE^\dagger(\eta)\boxtimes\BZ_p^\times$ is stable under $w_\delta$.
\end{prop}
\begin{proof}
It suffices to treat the case $\eta=1$. By \cite[Lemme V.1.16]{Colmez-Langlands}, $\SE^\dagger\boxtimes\BZ_p^\times$ is a free $\SE^\dagger(\Gamma)$-module of rank 1 generated by $1+T$. By \cite[B.2.8]{P00}, $\SR^+\boxtimes\BZ_p^\times$ is a free $\SR^+(\Gamma)$-module also generated by $1+T$. We thus deduce that $\SE^+\boxtimes\BZ_p^\times=(\SE^\dagger)^{\psi=0}\cap(\SR^+)^{\psi=0}$ is a free $\SE^+(\Gamma)$-module generated by $1+T$. The last assertion follows by (\ref{eq:anti2}) and the fact that $\SE^+\boxtimes\BZ_p^\times$ is stable under $w_\delta$.
\end{proof}

\subsection{Construction of the correspondence}

We define
$$ D\boxtimes_\delta \RP^1 =\left\{ z=(z_1, z_2)\in D\times D \ | \
\Res_{\BZ_p^\times}(z_2)=w_\delta(\Res_{\BZ_p^\times}(z_1))
\right\},$$
and we equip $D\boxtimes_{\delta} \RP^1$ with the subspace topology of $D\times D$.
\begin{prop} \label{prop:GL2action}
 There exists a unique
continuous action of $\GL_2(\BQ_p)$ on $D\boxtimes_\delta \RP^1$ satisfying the
following conditions:
\begin{enumerate}
\item $w (z_1, z_2) =(z_2, z_1)$;
\item if $a\in \BQ_p^\times$, then $\wvec{a}{0}{0}{a} (z_1, z_2) =(\delta(a)z_1, \delta(a)z_2)
$;
\item if $a\in \BZ_p^\times$, then $\wvec{a}{0}{0}{1} (z_1, z_2)
=(\sigma_a (z_1), \delta(a)\sigma_{a^{-1}}(z_2)) $;
\item if $z=(z_1, z_2)$ and if $z'=\wvec{p}{0}{0}{1}(z)$, then
$\Res_{p\BZ_p}z'=\varphi(z_1)$ and
$\Res_{\BZ_p}(wz')=\delta(p)\psi(z_2)$;
\item for $b\in p\BZ_p$, if $z=(z_1, z_2)$ and if $z'=\wvec{1}{b}{0}{1}
z$, then $\Res_{\BZ_p}z'=\wvec{1}{b}{0}{1}z_1$ and $\Res_{p\BZ_p}
(wz') =u_b(\Res_{p\BZ_p}(z_2)) $, where
$$ u_b=\delta(1+b)\wvec{1}{-1}{0}{1}\circ w_\delta\circ
\wvec{(1+b)^{-2}}{b(1+b)^{-1}}{0}{1}\circ w_\delta \circ
\wvec{1}{1/(1+b)}{0}{1}. $$
\end{enumerate}
\end{prop}
\begin{proof}
It is easy to see that the matrices given in the proposition generate $\GL_2(\BQ_p)$. This implies the uniqueness. The existence follows by \cite[Proposition II.1.8]{Colmez-Langlands}. (Its proof applies to our more general situation.)
\end{proof}

We extend $\{ \cdot ,\cdot \}: \check{D}\times D\rightarrow L$
to a pairing $\{ \cdot ,\cdot   \}_{\RP^1}: (\check{D}\boxtimes_{\delta^{-1}}\RP^1)\times
(D\boxtimes_\delta\RP^1)\ra L $ by setting
\begin{equation}\label{eq:pairing}
\{ (z_1,z_2), (z'_1,z'_2) \}_{\RP^1} =
\{ z_1,z'_1\}+\{\Res_{p\BZ_p}(z_2), \Res_{p\BZ_p}(z'_2)\} .
\end{equation}
\begin{prop} \label{prop:pairing}
The pairing $\{ \cdot,\cdot \}_{\RP^1}: (\check{D}\boxtimes_{\delta^{-1}}\RP^1)\times
(D\boxtimes_\delta\RP^1)\ra L $ is perfect and  $\GL_2(\BQ_p)$-equivariant.
\end{prop}
\begin{proof}
This follows immediately from \cite[Th\'eor\`eme II.3.1]{Colmez-Langlands}.
\end{proof}

Now Let $D$ be of rank two. Then
$\det D$ is of the form $\mathcal{O}_\SE(\delta'_D)$ or $\SE(\delta'_D)$ for some $\delta'_D\in\widehat{\mathscr{T}}(L)$. Let
$\delta_D=\epsilon^{-1}\delta'_D$, and we denote $w_{\delta_D}, D\boxtimes_{\delta_D}\RP^1$ by $w_D, D\boxtimes \RP^1$ for simplicity. For $z=(z_1,z_2)\in D\boxtimes \RP^1$, set $\Res_{\BZ_p}(z_1,z_2)=z_1$. We then define
$$D^\natural \boxtimes \RP^1=\{ z\in D\boxtimes \RP^1 \ |\
\Res_{\BZ_p}(\wvec{p^n}{0}{0}{1}z)\in D^\natural, \forall n\in\mathbb{N} \}.$$

\begin{thm} $($\cite[Th\'eor\`eme II.3.1]{Colmez-Langlands}$)$
Let $D$ be an irreducible rank $2$ \'etale $(\varphi,\Gamma)$-module over $\SO_\SE$. Then the following hold:
\begin{enumerate}
\item The submodule $D^\natural \boxtimes \RP^1$ of $D\boxtimes
\RP^1$ is stable under the action of $\GL_2(\BQ_p)$.
\item The quotient $\GL_2(\BQ_p)$-representation $\Pi(D)=D\boxtimes \RP^1/ D^\natural \boxtimes
\RP^1$ is an object of $\Rep_{\mathcal{O}_L}\GL_2(\BQ_p)$, and has central
character $\delta_D$.  The continuous $\GL_2(\BQ_p)$-representation $D^\natural \boxtimes \RP^1$ is naturally
isomorphic to $\Pi(D)^*\otimes \delta_D$. Hence we have the following
exact sequence
$$ 0 \longrightarrow\Pi(D)^* \otimes \delta_D \longrightarrow
 D \boxtimes \RP^1 \longrightarrow\Pi(D) \longrightarrow 0.$$
\end{enumerate}
\end{thm}
\noindent Note that $\check{D}\cong D(\delta_D^{-1})$ because $D$ is of rank two. Hence $\Pi(D)^*\otimes\delta_{D}$ is isomorphic to
$$(\Pi(\check{D})\otimes\delta_D)^*\otimes\delta_D=\Pi(\check{D})^*.$$
\begin{cor}
If $D$ is an irreducible \'etale $(\varphi,\Gamma)$-module of rank 2 over
$\SE$, then $D^\natural \boxtimes \RP^1$ is stable under
the $\GL_2(\BQ_p)$-action and the quotient representation $\Pi(D)=D\boxtimes
\RP^1/D^\natural \boxtimes \RP^1$ is an admissible unitary representation of $\GL_2(\BQ_p)$. Moreover, $D^\natural \boxtimes \RP^1$ is
naturally isomorphic to the contragredient representation $\Pi(\check{D})^*$ .
\end{cor}

\subsection{Locally analytic vectors}
Now suppose $D$ is a twist of an irreducible rank $2$ \'etale $(\varphi,\Gamma)$-module over $\SE$. The following proposition follows immediately from $($\cite[Lemme V.2.4]{Colmez-Langlands}$)$.
\begin{prop}
 $D^\dagger\boxtimes
\BZ_p^\times$ is stable under the action of $w_D$.
\end{prop}

By \cite[Th\'eor\`eme V.1.12(iii)]{Colmez-Langlands}, we have $D_\rig\boxtimes
\BZ_p^\times=\SR(\Gamma)\otimes_{\SE^\dagger(\Gamma)}(D^\dagger\boxtimes
\BZ_p^\times)$. Then for $\Delta=\SR(\eta), *=\delta$, or $\Delta=D_\rig,*=\delta_D$,
we extend the $w_*$-action to $\Delta\boxtimes
\BZ_p^\times$ by setting
\[
w_*(\lambda \otimes z) =\iota_{*}(\lambda ) \otimes w_*(z).
\]
For $\Delta=\SE^\dagger(\eta),\SR^+(\eta),\SR(\eta), *=\delta$, or $\Delta=D^\dagger,D_\rig,*=\delta_D$, we set
\[
\Delta\boxtimes_{*}\RP^1=
\{ (z_1,z_2)\in \Delta\times \Delta,
\Res_{\BZ_p^\times}(z_2)=w_*(\Res_{\BZ_p^\times}(z_1)) \},
\]
and we equip $\Delta\boxtimes_{*}\RP^1$ with the subspace topology of $\Delta\times\Delta$. Henceforth we denote $D^\dagger\boxtimes_{\delta_D}\RP^1$, $D_\rig\boxtimes_{\delta_D}\RP^1$ by $D^\dagger\boxtimes\RP^1$, $D_\rig\boxtimes\RP^1$ for simplicity.

By Proposition
\ref{prop:GL2action}, it is clear that both $\SE^\dagger(\eta)\boxtimes_\delta\RP^1$ and $D^\dagger\boxtimes\RP^1$ are stable under $\GL_2(\BQ_p)$. Since $D^\dagger$ is dense in $D_\rig$ for any $(\varphi,\Gamma)$-module $D$ over $\SE$, we extend the $\GL_2(\BQ_p)$-actions on $\SE^\dagger(\eta)\boxtimes_\delta\RP^1$ and $D^\dagger\boxtimes\RP^1$ to continuous $\GL_2(\BQ_p)$-actions on $\SR(\eta)\boxtimes_\delta\RP^1$ and $D_\rig\boxtimes \RP^1$ which satisfy the formulas listed in Proposition \ref{prop:GL2action} by continuity. This yields the following proposition.

\begin{prop} \label{prop:rank1-GL2action}
There exists a unique continuous
action of $\GL_2(\BQ_p)$ on $\Delta\boxtimes_* \RP^1$ satisfying the
formulas listed in Proposition \ref{prop:GL2action}.
\end{prop}

For $\Delta=\SR(\eta), *=\delta$, or $\Delta=D_\rig,*=\delta_D$, we set the pairing $\{\cdot,\cdot\}_{\RP^1}:\check{\Delta}\boxtimes_{*^{-1}}\RP^1\times
\Delta\boxtimes_{*}\RP^1\rightarrow L$
by formula (\ref{eq:pairing}).

\begin{prop} \label{prop:pairing,rank 1}
The pairing $\{\cdot,\cdot\}_{\RP^1}: \check{\Delta}\boxtimes_{*^{-1}}\RP^1\times
\Delta\boxtimes_{*}\RP^1\rightarrow L$ is perfect and $\GL_2(\BQ_p)$-equivariant.
\end{prop}

\begin{proof}
The restriction of $\{\cdot,\cdot\}_{\RP^1}$ on $\SE^\dagger(\check{\eta})\boxtimes_{\delta^{-1}}\RP^1\times
\SE^\dagger(\eta)\boxtimes_{\delta}\RP^1$ or $\check{D}^\dagger\boxtimes\RP^1\times D^\dagger\boxtimes\RP^1$ is $\GL_2(\BQ_p)$-equivariant by Proposition \ref{prop:pairing}. Hence $\{\cdot,\cdot\}_{\RP^1}$ itself is $\GL_2(\BQ_p)$-equivariant by the density of $\SE^\dagger(\eta)\boxtimes_{\delta}\RP^1$ or $D^\dagger\boxtimes\RP^1$. The perfectness of $\{\cdot,\cdot\}_{\RP^1}$ follows from the perfectness of $\{\cdot,\cdot\}$ on $\check{\Delta}\times \Delta$ and $\check{\Delta}\boxtimes p\BZ_p\times\Delta\boxtimes p\BZ_p$.
\end{proof}

If $D$ is \'etale, it follows by \cite[Corollaire II.7.2]{Treillis} that $D^\natural\subset D^\dagger$; hence we may view $D^\natural\boxtimes\RP^1$ as a submodule of $D_\rig\boxtimes\RP^1$. Colmez shows that the inclusion  $\Pi(\check{D})^*=D^\natural\boxtimes\RP^1\subset D_\rig\boxtimes \RP^1 $ extends naturally to a $\GL_2(\BQ_p)$-equivariant embedding $(\Pi(\Check{D})_\an)^*\hookrightarrow D_\rig\boxtimes \RP^1 $, and he further shows that the image of this embedding, which is denoted by $D^\natural_\rig\boxtimes \RP^1$, is the orthogonal complement of $\check{D}^\natural\boxtimes\RP^1$ under the pairing $\{\cdot,\cdot\}_{\RP^1}:\check{D}_\rig\boxtimes\RP^1\times D_\rig\boxtimes \RP^1\ra L$ (\cite[Remarque V.2.21(ii)]{Colmez-Langlands}). The following proposition is the key ingredient for our determination of locally analytic vectors of unitary principal series.

\begin{prop}\label{prop:Drig-orthogonal}$($\cite[Remarque V.2.21(i)]{Colmez-Langlands}$)$
$D^\natural_\rig\boxtimes \RP^1 $ and $\check{D}^\natural_\rig\boxtimes \RP^1$ are orthogonal complements of each other under the pairing $\{\cdot ,\cdot\}_{\RP^1}:\check{D}_\rig\boxtimes\RP^1\times D_\rig\boxtimes \RP^1\ra L$.
\end{prop}

\section{Unitary principal series and 2-dimensional trianguline representations}

\subsection{2-dimensional irreducible trianguline representations}

A $(\varphi,\Gamma)$-module over $\SR$ is called {\it triangulable}
if it is a successive extensions of rank
1 $(\varphi,\Gamma)$-modules over $\SR$; an $L$-linear representation $V$ of $G_{\BQ_p}$ is called
{\it trianguline} if $\RD_\rig(V)$ is triangulable.

\begin{prop} $($\cite[Proposition 3.1]{Colmez08}$)$\label{prop:one-rank-class}
If $D$ is a rank $1$ $(\varphi,\Gamma)$-module over $\SR$, then there exists a unique $\delta\in
\widehat{\mathscr{T}}(L)$ such that $D$ is isomorphic to $\SR(\delta)$.
\end{prop}

It follows that if $V$ is a $2$-dimensional irreducible trianguline representation, then $\RD_{\rig}(V)$ sits in a short exact sequence
\begin{equation}\label{eq:exact-sq-tr}
0\longrightarrow\SR(\delta_1)\longrightarrow\RD_\rig(V)\longrightarrow\SR(\delta_2)\longrightarrow0
\end{equation}
for some $\delta_1,\delta_2\in\widehat{\mathscr{T}}(L)$. Furthermore, we have that $(\ref{eq:exact-sq-tr})$ is non-split by Kedlaya's slope theory. Therefore $V$ is uniquely determined by the triple $(\delta_1,\delta_2,c)$ where $$c\in\mathrm{Proj}(\mathrm{Ext}^1(\SR(\delta_2),\SR(\delta_1)))=\mathrm{Proj}(H^1(\SR(\delta_1\delta_2^{-1})))$$ is the element representing the extension $(\ref{eq:exact-sq-tr})$. Let
\[
\SS=\{(\delta_1,\delta_2,c)|\delta_i\in\widehat{\mathscr{T}}(L), c\in\mathrm{Proj}(H^1(\SR(\delta_1\delta_2^{-1})))\}
\]
be the set of all such triples; then each element of $\SS$ naturally defines a rank 2 triangulable $(\varphi,\Gamma)$-module: the non-split extension of $\SR(\delta_2)$ by $\SR(\delta_1)$ defined by $c$. In the rest of this section we assume $p>2$.
The following calculation is due to Colmez.
\begin{prop}$($\cite[Th\'eor\`{e}me 2.9]{Colmez08}$)$
\begin{enumerate}
\item[(i)]If $\delta=x^{-k}$ or $x^{k+1}|x|$ for $k\in\mathbb{N}$, then $\dim_LH^1(\SR(\delta))=2$.
\item[(ii)]Otherwise, $\dim_L(\SR(\delta))=1$.
\end{enumerate}
\end{prop}
\noindent Colmez further specifies an explicit basis of $H^{1}(\SR(\delta))$ in case $(\mathrm{i})$, and identifies $\mathrm{Proj}(H^1(\SR(\delta)))$ with $\RP^1(L)$ via this basis.
By this identification, we may write any $s\in\SS$ as $s=(\delta_1, \delta_2, \SL)$
where $\SL\in\RP^1(L)$ if $\delta_1\delta_2^{-1}=x^{-k}$ or $x^{k+1}|x|$ for some $k\in\mathbb{N}$, or $\SL=\infty$ otherwise. Let $D(s)$ denote the rank $2$ triangulable $(\varphi,\Gamma)$-module defined by $s$.  We set
$\check{s}\in\SS$ to be the triple $(\check{\delta_1},\check{\delta_2},\SL)=(\epsilon \delta_2^{-1},\epsilon\delta_1^{-1},\SL)$. Then
$\check{D}(s)$ is isomorphic to $D(\check{s})$ under Colmez's identification.

Let $\SS_*$ be the set of all $s=(\delta_1, \delta_2, \SL)\in \SS$
such that
$$v_p(\delta_1(p))+v_p(\delta_2(p))=0, \hskip 10pt
v_p(\delta_1(p))> 0.\hskip 10pt $$
For $s\in\SS_*$, set
$u(s)=v_p(\delta_1(p)) =-v_p(\delta_2(p)) ,
w(s)=w(\delta_1)-w(\delta_2), \delta_s=\delta_1\delta_2^{-1}(x|x|)^{-1}$, and we define
\begin{eqnarray*} \SS_*^\ng &=& \{ s\in\SS_* \ |\ w(s) \text{ is not
a positive integer } \},
\\
\SS_*^\cris &=& \{ s\in \SS_* \ |\ w(s) \text{ is a positive integer }
, u(s)<w(s), \SL=\infty \},
\\
\SS_*^\st &=& \{ s\in \SS_* \ |\ w(s) \text{ is a positive integer },
u(s)<w(s), \SL\neq \infty \},
\end{eqnarray*}
and
$\SS_\irr=\SS_*^\ng\coprod \SS_*^\cris\coprod \SS_*^\st$. Note that if $s\in \SS_*^\st$, we must have $\delta_s=x^{k-1}$ for some $k\in\BZ_+$.
\begin{prop}$($\cite[Proposition 3.5,3.7]{Colmez08}$)$\label{prop:class-rank2}
If $s\in\SS_{\irr}$, then $D(s)$ is \'etale, and the corresponding $L$-linear representation, which we denote by $V(s)$, is irreducible. Conversely, every $2$-dimensional irreducible trianguline representation is isomorphic to $V(s)$ for some $s\in\SS_{\irr}$. Moreover, for $s=(\delta_1,\delta_2, \SL), s'=(\delta'_1,\delta'_2,\SL')\in\SS_\irr$, if $\delta_1=\delta'_1$, then $V(s)\cong V(s')$ if and only if
$\delta_2=\delta'_2$ and $\SL=\SL'$; if $\delta_1\neq \delta'_1$, then $V(s)\cong V(s')$ if and only
if $s,s'\in \SS_*^\cris$ and $\delta'_1=x^{w(s)}\delta_2, \
\delta'_2=x^{-w(s)}\delta_1$.
\end{prop}

We call $s\in\SS_{\irr}$ \emph{exceptional} if $s\in\SS_*^{\mathrm{cris}}$ and $V(s)$ is not Frobenius semi-simple; this is equivalent to $\delta_s=x^{k-1}|x|^{-1}$ for some $k\in\mathbb{Z}_+$.

\subsection{Unitary principal series}
Throughout this subsection, let $s=(\delta_1,\delta_2,\SL)\in\SS_\irr$. Let $\SC^u(\RP^1(\delta))$ be the $L$-vector space of $\SC^u$ functions $f:\BQ_p\rightarrow L$ such that $\delta(x)f(1/x)|_{\BQ_p-\{0\}}$ extends to a $\SC^u$-function on $\BQ_p$. In other words,
$\SC^{u(s)}(\RP^1(\delta_s))$ is the $L$-vector space of functions
$f: \BQ_p\rightarrow L$ satisfying:

$\bullet$ $f|_{\BZ_p}$ is of class $\SC^{u(s)}$.

$\bullet$ $\delta_s(z)f(1/z)|_{\BZ_p-\{0\}}$ extends to
a $\SC^{u(s)}$-function on $\BZ_p$.

\noindent We thus have an isomorphism
$$ \SC^{u(s)}(\RP^1(\delta_s)) \simeq \SC^{u(s)}(\BZ_p,L)\oplus \SC^{u(s)}(\BZ_p,L),
\hskip 10pt f\mapsto (f_1, f_2)$$ where $f_1(z)=f(pz)$ and
$f_2$ is the extension of $\delta_s(z)f(1/z)$. By this isomorphism, we may equip $\SC^{u(s)}(\RP^1(\delta_s))$
with a Banach space structure by defining
$$ ||f|| = \max\Big(\ ||f_1||_{\SC^{u(s)}}, \ ||f_2||_{\SC^{u(s)}} \ \Big) . $$
We define a $\GL_2(\BQ_p)$-representation $B(s)$ on
$\SC^{u(s)}(\RP^1(\delta_s))$ by
$$\left(\wvec{a}{b}{c}{d} \star_s f \right)
(x)= \delta_2(ad-bc) \delta_s(-cx+a) f
\left(\frac{dx-b}{-cx+a}\right);$$ then $B(s)$ is a Banach space
representation. We define a subspace $M(s)$ of $B(s)$ as below:

$\bullet$ If $\delta_s\neq x^k$ for any $k\in\BN$, we
define $M(s)$ to be the subspace generated by $\{x^i|0\leq
i<u(s)\}$ and $\{(x-a)^{-i}\delta_s(x-a)|a\in \BQ_p, 0\leq
i<u(s)\}$.

$\bullet$ If $\delta_s=x^{k-1}$ for some $k\in\BZ_+$, let $M(s)'$ be the space of functions of the form
$$ f=\sum_{u\in U} \lambda_u (x-a_u)^{j_u}\log_\SL(x-a_u)$$
where $U$ is a finite set, $j_u$ are integers between $[\frac{k+1}{2}]$ and $k$, $\lambda_u\in L$
and $a_u\in \BQ_p$ such that $\deg(\sum_{u\in
U}\lambda_u(x-a_u)^{j_u})<u(s)$.
By  \cite[Lemme 3.3.2 ]{Breuil-com}, $M(s)'$ is a subspace of
$B(s)$. We define $M(s)$ to be the subspace generated by $M(s)'$ and
$x^i$ for $0\leq i\leq k-1$.

\noindent An easy computation shows that $M(s)$ is stable under $\GL_2(\BQ_p)$ in both cases. We set $\Pi(s)=B(s)/\widehat{M}(s)$ where $\widehat{M}(s)$ is the closure of $M(s)$ in $B(s)$.

Let $D^\natural(s)$ denote $(\D(V(s)))^\natural$ for simplicity. We fix a standard basis $e_2$ of $\SR(\delta_2)$.
For any $z\in D^\natural(s) \boxtimes \RP^1$, suppose the image of $\Res_{\BZ_p}(\wvec{p^n}{0}{0}{1}z)$ in $\SR(\delta_2)$ is
$z_2^{(n)}e_2$. The following theorem follows by \cite[Th\'eor\`{e}me IV.4.12]{Colmez-Langlands}.
\begin{thm} \label{thm:Amice-rank2}
For $s\in\SS_\irr$ non-exceptional and $z\in D^\natural(s) \boxtimes \RP^1$, there exists $\mu_z\in \SD_{u(s)}(\BQ_p)$ such that $$\mathscr{A}^{(n)}(\mu_z)=(\delta_2(p))^{-n}z_2^{(n)}.$$
Furthermore, the map
$z\mapsto \mu_z$
is a $\GL_2(\BQ_p)$-equivariant topological isomorphism
from $D^\natural(s) \boxtimes \RP^1$ to $\Pi(\check{s})^*$.
\end{thm}
We denote the converse of this isomorphism by $\mathscr{A}_s$.

\section{Locally analytic principal series and rank $1$ $(\varphi,\Gamma)$-modules}

\subsection{Locally analytic principal series}\label{ss:locally-principal}
 For any $\delta\in\widehat{\mathscr{T}}(L)$, we denote by $\LA(\RP^1(\delta))$ the $L$-vector space of locally analytic functions $f:\BQ_p\rightarrow L$ such that $\delta(x)f(1/x)|_{\BQ_p-\{0\}}$ extends to a
locally analytic function on $\BQ_p$. As in the case of $\SC^u(\RP^1(\delta))$, for any $f\in\LA(\RP^1(\delta))$, if we set $f_1(pz)=f|_{p\BZ_p}$ and $f_2$ to be the extension of $\delta_s(x)f(1/x)|_{\BZ_p-\{0\}}$, then the map $f\mapsto f_1\oplus f_2$ induces an isomorphism from $\LA(\RP^1(\delta))\cong\LA(\BZ_p,L)\oplus\LA(\BZ_p,L)$. We then equip $\LA(\RP^1(\delta))$ with the topology induced from $\LA(\BZ_p,L)\oplus\LA(\BZ_p,L)$.

For any pair $(\delta_1,\delta_2)\in\widehat{\mathscr{T}}(L)\times\widehat{\mathscr{T}}(L)$, let $\widetilde{\Sigma}(\delta_1,\delta_2)$ denote the locally analytic parabolic induction
\begin{equation*}
\begin{split}
\left(\Ind^{\GL_2(\BQ_p)}_{\mathrm{B}(\BQ_p)}\delta_2\otimes\delta_1\epsilon^{-1}\right)_\an&=\{\text{locally analytic functions}\ F:\GL_2(\BQ_p)\ra L\ \text{such} \\
&\text{that}\ F(bg)=(\delta_2\otimes\delta_1\epsilon^{-1})(b)F(g)\ \text{for all}\ b\in\mathrm{B}(\BQ_p) \},
\end{split}
\end{equation*}
which is equipped with the left $\GL_2(\BQ_p)$-action $(gF)(g')=F(g'g)$ for any $g,g'\in \GL_2(\BQ_p)$. Put $\delta=\delta_1\delta_2^{-1}\epsilon^{-1}$. We may identify the underlying topological space of $\widetilde{\Sigma}(\delta_1,\delta_2)$ with $\LA(\RP^1(\delta))$ by the map
$$F\mapsto f(x):=F(\wvec{0}{1}{-1}{x})$$
for any $F\in\widetilde{\Sigma}(\delta_1,\delta_2)$. In addition, the corresponding $\GL_2(\BQ_p)$-action on $\LA(\RP^1(\delta))$ is given by the formula
\begin{equation}
   \left(\wvec{a}{b}{c}{d} \cdot f \right)
(x)= \delta_2(ad-bc) \delta(-cx+a) f
\left(\frac{dx-b}{-cx+a}\right).
\end{equation}


If  $k=w(\delta_1\delta_2^{-1})=w(\delta)+1$ is a positive integer,  
then the $k$-th differential map
$$ I_k: \LA(\RP^1(\delta))\rightarrow \LA(\RP^1(x^{-2k}\delta)),
\hskip 10pt f(x) \mapsto \left(\frac{\mathrm{d}}{\mathrm{d}
x}\right)^{k}f(x),$$
induces an intertwining between $\widetilde{\Sigma}(\delta_1,\delta_2)$ and $\widetilde{\Sigma}(x^{-k}\delta_1,x^k\delta_2)$. The kernel of $I_k$, which consists of locally polynomial functions of degree $\leq k-1$, is isomorphic to
\begin{equation}\label{eq:lalg}
(\delta_2\circ\det)\otimes\Sym^{k-1}L^2 \otimes \Ind_{\RB(\BQ_p)}^{\GL_2(\BQ_p)}
(1\otimes (x^{-k+1}\delta) )_{\mathrm{sm}}
\end{equation}
as a locally analytic representation. Moreover, if $\delta=x^{k-1}$, the $L$-vector subspace generated by $\{x^i|0\leq i\leq k-1\}$ is $\GL_2(\BQ_p)$-invariant, and is isomorphic to $(\delta_2\circ\det)\otimes\Sym^{k-1}L^2$ as a $\GL_2(\BQ_p)$-representation. The quotient of $\ker I_k$ by this subspace is isomorphic to $(\delta_2\circ\det)\otimes\Sym^{k-1}L^2 \otimes\mathrm{St}$.

We define
\begin{equation}\label{eq:Sigma}
\Sigma(\delta_1, \delta_2)=\left\{
         \begin{array}{ll}
         \widetilde{\Sigma}(\delta_1, \delta_2)/(\delta_2\circ\det)\otimes\Sym^{k-1}L^2 & \text{if $\delta=x^{k-1}$ for some integer $k\geq1$}; \\
          \widetilde{\Sigma}(\delta_1, \delta_2) & \text{otherwise}.
         \end{array}
       \right.
\end{equation}
The following proposition, which follows by the main results of \cite{ST-an}, \cite{ST-an2}, determines the Jordan-H\"{o}lder series of $\Sigma(\delta_1,\delta_2)$.
\begin{prop}\label{prop:ST-an}
With notations as above, the following are true.
\begin{enumerate}
\item[(i)]If $w(\delta)\notin\mathbb{N}$, then $\Sigma(\delta_1,\delta_2)=\widetilde{\Sigma}(\delta_1,\delta_2)$ is a topological irreducible locally analytic representation of $\GL_2(\BQ_p)$.
\item[(ii)]If $w(\delta)\in\mathbb{N}$ and $\delta\neq x^{k-1}$, then $I_k$ is surjective, and $\Sigma(\delta_1,\delta_2)=\widetilde{\Sigma}(\delta_1,\delta_2)$ is a non-split extension of $\widetilde{\Sigma}(x^{-k}\delta_1,x^{k}\delta_2)$ by $(\delta_2\circ\det)\otimes\Sym^{k-1}L^2 \otimes \Ind_{\RB(\BQ_p)}^{\GL_2(\BQ_p)}
(1\otimes (x^{-k+1}\delta) )_{\mathrm{sm}}$, and both $\widetilde{\Sigma}(x^{-k}\delta_1,x^{k}\delta_2)$ and $(\delta_2\circ\det)\otimes\Sym^{k-1}L^2 \otimes \Ind_{\RB(\BQ_p)}^{\GL_2(\BQ_p)}
(1\otimes (x^{-k+1}\delta) )_{\mathrm{sm}}$ are topological irreducible.
\item[(iii)]If $\delta=x^{k-1}$ for some integer $k\geq1$, then $\Sigma(\delta_1,\delta_2)$ is a non-split extension of
$\widetilde{\Sigma}(x^{-k}\delta_1,x^{k}\delta_2)$ by $(\delta_2\circ\det)\otimes\Sym^{k-1}L^2 \otimes\mathrm{St}$, and both $\widetilde{\Sigma}(x^{-k}\delta_1,x^{k}\delta_2)$ and $(\delta_2\circ\det)\otimes\Sym^{k-1}L^2 \otimes\mathrm{St}$ are topological irreducible.
\end{enumerate}
\end{prop}

\subsection{$\tilde{\Sigma}(\eta^{-1}\epsilon,
\delta^{-1}\eta)^*\cong\SR^+(\eta)\boxtimes_\delta \RP^1$}
For any $\delta_1,\delta_2\in\widehat{\mathscr{T}}(L)$, let $\GL_2(\BQ_p)$ acts on
$\widetilde{\Sigma}(\delta_1,
\delta_2)^*$
by the formula $\langle f, g\cdot\mu \rangle=\langle g^{-1}\cdot f,
\mu\rangle$ for any $f\in
\widetilde{\Sigma}(\delta_1,
\delta_2)$, $\mu\in
\widetilde{\Sigma}(\delta_1,
\delta_2)^*$
and $g\in\GL_2(\BQ_p)$.
Thus by $(5.1)$, we have
\[
(wf)(x)=\eta(-1)\delta\eta^{-2}(x)f(1/x)
\]
for any $f\in\widetilde{\Sigma}(\eta^{-1}\epsilon,
\delta^{-1}\eta)$. Therefore, by the description of $\LA(\RP^1(\delta))$ given in \S5.1, we see that the map $\mu\mapsto(\mu|_{\BZ_p}, w\mu|_{\BZ_p})$ is a homeomorphism from $\widetilde{\Sigma}(\eta^{-1}\epsilon,
\delta^{-1}\eta)^*$ to
\begin{equation}\label{eq:w-action}
\{(\mu_1,\mu_2)\in\mathscr{D}(\BZ_p,L)\oplus\mathscr{D}(\BZ_p,L)
|\int_{\BZ_p^\times} f(x)\mu_2=\int_{\BZ_p^\times}\eta(-1)(\delta\eta^{-2})(x)f(1/x)\mu_1\},
\end{equation}
where the latter object is equipped with the subspace topology of $\mathscr{D}(\BZ_p,L)\oplus\mathscr{D}(\BZ_p,L)$.

We fix a standard basis $e_\eta\in \SE^+(\eta)$.
\begin{lem}\label{lem:w=wdelta} ~
$\mathscr{A}(w\mu|_{\BZ_p^\times})\otimes e_\eta=w_{\delta}(\mathscr{A}(\mu|_{\BZ_p^\times})\otimes e_\eta)$ for any $\mu\in\widetilde{\Sigma}(\eta^{-1}\epsilon,
\delta^{-1}\eta)^*$
\end{lem}
\begin{proof}
The case $\mathscr{A}(\mu|_{\BZ_p^\times})\in\SE^+\boxtimes\BZ_p^\times$ follows directly
from (\ref{eq:w_delta}) and (\ref{eq:w-action}). Since $\SE^+$ is dense in $\SR^+$, we deduce the case $\mathscr{A}(\mu|_{\BZ_p^\times})\in\SR^+\boxtimes\BZ_p^\times$ by the continuity of $w$ and $\mathscr{A}$.
We then conclude the general case by Proposition \ref{prop:rank1}.
\end{proof}
As a consequence, $\mathscr{A}_{\delta,\eta}(\mu)=(\mathscr{A}(\mu|_{\BZ_p})\otimes e_\eta, \mathscr{A}(w\mu|_{\BZ_p})\otimes e_\eta)$ is an element of $\SR^+(\eta)\boxtimes_\delta \RP^1$.

\begin{prop} \label{prop:map-A} ~
The map $\mathscr{A}_{\delta,\eta}:
\widetilde{\Sigma}(\eta^{-1}\epsilon,
\delta^{-1}\eta)^*
\rightarrow \SR^+(\eta)\boxtimes_\delta \RP^1$
is a $\GL_2(\BQ_p)$-equivariant topological
isomorphism.
\end{prop}
\begin{proof}
By the description of $\widetilde{\Sigma}(\eta^{-1}\epsilon,
\delta^{-1}\eta)^*$ given in (\ref{eq:w-action}), one sees easily that $\mathscr{A}_{\eta,\delta}$ is an embedding.
On the other hand, for any $z=(z_1\otimes e_{\eta}, z_2\otimes e_{\eta})\in\SR^+(\eta)\boxtimes_\delta \RP^1$, if we put $\mu=\mathscr{A}^{-1}(z_1)+w\mathscr{A}^{-1}(\Res_{p\BZ_p}(\mu_2))$, then $\mathscr{A}_{\eta,\delta}(\mu)=z$. Hence
$\mathscr{A}_{\eta,\delta}$ is a topological isomorphism.

To prove that $\mathscr{A}_{\eta,\delta}$ is $\GL_2(\BQ_p)$-equivariant, we only need to
show
\begin{equation}
\mathscr{A}_{\eta,\delta}(g\cdot\mu)=g\cdot\mathscr{A}_{\eta,\delta}(\mu)
\end{equation}
for $(1)$ $g=\wvec{0}{1}{1}{0}$; $(2)$ $g=\wvec{a}{0}{0}{a},a\in\BQ_p^\times$; $(3)$ $g=\wvec{a}{0}{0}{1}, a\in\BZ_p^\times$; $(4)$ $g=\wvec{p}{0}{0}{1}$; $(5)$ $g=\wvec{1}{b}{0}{1}, b\in p\BZ_p$.

Case $(1)$ is trivial. Both $\widetilde{\Sigma}(\eta^{-1}\epsilon,\delta^{-1}\eta)^*$
and $\SR^+(\eta)\boxtimes_\delta \RP^1$ have central characters $\delta$; this proves $(2)$. For any $a\in\BZ_p^\times$, we have
\[
\int_{\BZ_p}f(x)(\wvec{a}{0}{0}{1})\mu)=\int_{\BZ_p}(\wvec{a^{-1}}{0}{0}{1}f(x))\mu=\int_{\BZ_p}\eta(a^{-1})f(ax)\mu=
\int_{\BZ_p}f(x)(\eta(a^{-1})(\sigma_a(\mu)));
\]
this yields $(3)$. For case $(4)$, we have
\begin{equation*}
\begin{split}
\int_{p\BZ_p} f(x)(\wvec{p}{0}{0}{1}\mu)&=\int_{p\BZ_p}\wvec{p^{-1}}{0}{0}{1}(f(x) 1_{\BZ_p}(x))\mu\\
&=\int_{\BZ_p}\eta(p)f(px)\mu\\
&=\int_{\BZ_p}f(x)(\eta(p)\varphi(\mu|_{\BZ_p})),
\end{split}
\end{equation*}
yielding $\wvec{p}{0}{0}{1}\mu|_{p\BZ_p}=\eta(p)\varphi(\mu|_{\BZ_p})$. This implies
\begin{equation}\label{eq:map-A-3}
\Res_{p\BZ_p}(\mathscr{A}_{\delta,\eta}(\wvec{p}{0}{0}{1}\mu))=\varphi(\Res_{\BZ_p}(\mathscr{A}_{\delta,\eta}(\mu))).
\end{equation}
A similar computation shows that
\begin{equation}\label{eq:map-A-4}
\begin{split}
\Res_{\BZ_p}(w(\mathscr{A}_{\delta,\eta}&(\wvec{p}{0}{0}{1}\mu)))
=\Res_{\BZ_p}(\mathscr{A}_{\delta,\eta}(\wvec{1}{0}{0}{p}w\mu))
=\delta(p)\Res_{\BZ_p}(\mathscr{A}_{\delta,\eta}(\wvec{p^{-1}}{0}{0}{1}w\mu))\\
&=\delta(p)\psi(\Res_{\BZ_p}((\mathscr{A}_{\delta,\eta}(w\mu))))
=\delta(p)\psi(\Res_{\BZ_p}(w(\mathscr{A}_{\delta,\eta}(\mu)))).
\end{split}
\end{equation}
This proves $(4)$.

For case $(5)$, first note that
$$\int f(x)
\wvec{1}{b'}{0}{1}\mu=\int\wvec{1}{-b'}{0}{1}f(x)\mu=\int f(x+b')\mu.$$
This implies
\begin{equation}\label{eq:map-A-5}
\mathscr{A}(\wvec{1}{b'}{0}{1}\mu|_{U})=\wvec{1}{b'}{0}{1}\mathscr{A}(\mu|_{U})
\end{equation}
for any $b'\in U\subseteq\BZ_p$. Hence $\Res_{\BZ_p}(\mathscr{A}_{\delta,\eta}(\wvec{1}{b}{0}{1}\mu))
=\wvec{1}{b}{0}{1}\Res_{\BZ_p}(\mathscr{A}_{\delta,\eta}(\mu))$.
It remains to check that
\begin{equation}
\Res_{p\BZ_p}(\mathscr{A}_{\delta,\eta}(w\wvec{1}{b}{0}{1}\mu))=
u_b(\Res_{p\BZ_p}(\mathscr{A}_{\delta,\eta}(w\mu))),
\end{equation}
where
$$u_b=\delta(1+b)\wvec{1}{-1}{0}{1}\circ w_\delta \circ
\wvec{(1+b)^{-2}}{b(1+b)^{-1}}{0}{1}\circ w_\delta \circ \wvec{1}{1/(1+b)}{0}{1}.$$ By (\ref{eq:map-A-5}), Lemma \ref{lem:w=wdelta} and case (2),
\begin{equation*}
\begin{split}u_b(\Res_{p\BZ_p}(\mathscr{A}_{\delta,\eta}(w\mu)))&=\Res_{p\BZ_p}(\mathscr{A}_{\delta,\eta}(
\delta(1+b)\wvec{1}{-1}{0}{1}w\wvec{(1+b)^{-2}}{b(1+b)^{-1}}{0}{1}w \wvec{1}{1/(1+b)}{0}{1}w\wvec{1/(1+b)}{0}{0}{1/(1+b)}\mu))\\
&=\Res_{p\BZ_p}(\mathscr{A}_{\delta,\eta}(\wvec{0}{1}{1}{b}\mu))\\
&=\Res_{p\BZ_p}(\mathscr{A}_{\delta,\eta}(w\wvec{1}{b}{0}{1}\mu)).
\end{split}
\end{equation*}
This proves $(5.12)$.
\end{proof}

\section{Determination of locally analytic vectors} \label{sec:GL2-side}

\subsection{$\Sigma(s)$ and Emerton's conjecture}

We first recall the locally analytic representation $\Sigma(k,\SL)$ of $\GL_2(\BQ_p)$ which is originally constructed by Breuil (in the case $\SL\neq \infty$). We refer the reader to \cite[2.1]{Breuil-DS} and \cite[5.1]{Emerton-lg} for more details. Fix an integer $k\geq2$. Given $\SL\in\RP^1(L)$, let $\sigma(\SL)$ denote the representation of $\mathrm{B}(\BQ_p)$ on $L^2=Le_1\oplus Le_2$ defined by $\wvec{a}{b}{0}{d}e_1=e_1, \wvec{a}{b}{0}{d}e_2=e_1+(\log_{\SL}a-\log_{\SL}d)e_2$. One thus has a non-split extension
\begin{equation}\label{eq:sigma(L)}
0\longrightarrow 1\longrightarrow\sigma(\SL)
\longrightarrow1\longrightarrow 0.
\end{equation}
We put $\sigma(k,\SL)=\sigma(\SL)\otimes\chi_k$ where $\chi_k:\mathrm{B}(\BQ_p)\ra L^\times$ is the character $\wvec{a}{b}{0}{d}\mapsto|ad|^{\frac{k-2}{2}}d^{k-2}$. Twisting (\ref{eq:sigma(L)}) by $\chi_k$, and then taking locally analytic parabolic induction, one obtains an exact sequence of locally analytic representations
\begin{equation}\label{eq:Indsigma(k,L)}
0\longrightarrow(\Ind^{\GL_2(\BQ_p)}_{\mathrm{B}(\BQ_p)}\chi_k)_{\mathrm{an}}\longrightarrow (\Ind^{\GL_2(\BQ_p)}_{\mathrm{B}(\BQ_p)}\sigma(k,\SL))_{\mathrm{an}}
\stackrel{s_{\SL}}{\longrightarrow}(\Ind^{\GL_2(\BQ_p)}_{\mathrm{B}(\BQ_p)}\chi_k)_{\mathrm{an}}\longrightarrow 0.
\end{equation}
Note that $\chi_k=|x|^{\frac{k-2}{2}}\otimes x^{k-2}|x|^{\frac{k-2}{2}}$. Thus $(\Ind^{\GL_2(\BQ_p)}_{\mathrm{B}(\BQ_p)}\chi_k)_{\mathrm{an}}=\widetilde{\Sigma}( x^{k-1}|x|^{\frac{k}{2}},|x|^{\frac{k-2}{2}})$ which has $(|x|^{\frac{k-2}{2}}\circ\det)\otimes\Sym^{k-2}L^2$ as a subrepresentation following the discussion above. We define
\begin{equation}\label{eq:sigma(k,L)}
\Sigma(k,\SL)=s_{\SL}^{-1}((|x|^{\frac{k-2}{2}}\circ\det)\otimes\Sym^{k-2}L^2)
/(|x|^{\frac{k-2}{2}}\circ\det)\otimes\Sym^{k-2}L^2.
\end{equation}
One thus has an extension of locally analytic representations
\begin{equation} \label{eq:exact-sq-Sigma}
0\longrightarrow\Sigma(x^{k-1}|x|^{\frac{k}{2}},|x|^{\frac{k-2}{2}})\longrightarrow \Sigma(k,\SL)
\longrightarrow(|x|^{\frac{k-2}{2}}\circ\det)\otimes\Sym^{k-2}L^2\longrightarrow 0
\end{equation}

From now on, let $s=(\delta_1,\delta_2,\SL)\in\SS_\irr$. We define
\begin{equation}\label{eq:Sigma(s)}
\Sigma(s)=\left\{
         \begin{array}{ll}
         \Sigma(k+1,\SL)\otimes((\delta_2|x|^{\frac{2-k}{2}})\circ\det)& \text{if $w(s)=k$ is a positive integer and $\delta_s=x^{k-1}$}; \\
          \widetilde{\Sigma}(\delta_1, \delta_2) & \text{otherwise}.
         \end{array}
       \right.
\end{equation}
It follows that in the first case $\Sigma(s)$ sits in the exact sequence
\begin{equation}\label{eq:exact-sq-Sigma(s)}
0\longrightarrow\Sigma(\delta_1,\delta_2)\longrightarrow\Sigma(s)
\longrightarrow(\delta_2\circ\det)\otimes\Sym^{k-1}L^2 \longrightarrow 0.
\end{equation}%

Following \cite[2.2]{Breuil-DS}, we now give a geometric model of $\Sigma(s)$ in the first case. Let $\LA(\RP^1(x^{k-1},\SL))$ be the space of locally analytic
functions $H$ on $\BQ_p$ with values in $L$ such that
\begin{equation} \label{eqn:cond-H}
H(z)=z^{k-1}(\sum_{n=0}^{+\infty}\frac{a_n}{z^n})+P(z)\log_\SL(z),
\end{equation}
for $|z|\gg 0$, where $a_n\in L$, $P(z)$ is a polynomial of degree $\leq k-1$
with coefficients in $L$. Let $\GL_2(\BQ_p)$ act on this space by
{\small $$
 (\wvec{a}{b}{c}{d}\star_s H )(z) =
\delta_2(ad-bc)(-cz+a)^{k-1}
\left[H\left(\frac{dz-b}{-cz+a}\right) -\frac{1}{2}
P\left(\frac{dz-b}{-cz+a}\right)\log_\SL\left(\frac{ad-bc}{(-cz+a)^2}\right)\right].
$$}

\noindent Note that $\LA(\RP^1(x^{k-1}))$ is exactly the subspace
consisting of functions $H$ with $P=0$ in the expression
(\ref{eqn:cond-H}).
An easy computation shows that the $L$-vector subspace generated by $x^i$,
$0\leq i\leq k-1$, is $\GL_2(\BQ_p)$-invariant. We define $C(x^{k-1},\SL)$ to be the quotient of
$\LA(\RP^1(x^{k-1},\SL))$ by this subspace. It turns out that the resulting
representation of $\GL_2(\BQ_p)$ on $C(x^{k-1},\SL)$ is topologically isomorphic to $\Sigma(s)$, and the natural map
$\LA(\RP^1(x^{k-1}))\ra C(x^{k-1},\SL)$ gives rise the inclusion $\Sigma(\delta_1,\delta_2)\hookrightarrow\Sigma(s)$. We denote by $C(x^{k-1})$ the image of the map $\LA(\RP^1(x^{k-1}))\ra C(x^{k-1},\SL)$. Then the quotient $C(x^{k-1},\SL)/C(x^{k-1})$
is a $k$-dimensional $L$-vector space spanned by $1_{D(\infty,1)}\cdot x^n\log_\SL x$, $0\leq
n\leq k-1$ which is isomorphic to $(\delta_2\circ\det)\otimes\Sym^{k-1}L^2$. By this geometric model, one can show that (cf. \cite[Lemme 2.4.2]{Breuil-DS}) $(\delta_2\circ\det)\otimes \Sym^{k-1}L^2\otimes\mathrm{St}$ (resp. $(\delta_2\circ\det)\otimes\Sym^{k-1}L^2$) is the only topologically irreducible subrepresentation (resp. quotient representation) of $\Sigma(s)$. In particular, the extension (\ref{eq:exact-sq-Sigma(s)}) is non-split.

Although it is known to experts that there is a natural morphism $\Sigma(s)\ra\Pi(s)$ which realizes $\Pi(s)$ as the universal completion of $\Sigma(s)$, we can not find a reference for this result. For our purpose, we rephrase the work of Breuil and Emerton in the following proposition to construct the desired morphism. We first note that the natural inclusion $\LA(\RP^1(\delta_s))\subset\SC^{u(s)}(\RP^1(\delta_s))$ induces a $\GL_2(\BQ_p)$-equivariant continuous map
\[
\iota_s:\widetilde{\Sigma}(\delta_1,\delta_2)\ra\Pi(s), \quad f\mapsto \bar{f}.
\]

\begin{prop} \label{prop:injn-iotas}
For $s\in\SS_\irr$ non-exceptional, the $\GL_2(\BQ_p)$-equivariant continuous map $\iota_s:\widetilde{\Sigma}(\delta_1,\delta_2)\ra \Pi(s)$ induces an injection
$\iota_s:\Sigma(\delta_1,\delta_2)\ra \Pi(s)$. Moreover, in the case when $\delta_s=x^{k-1}$ for some positive integer $k$, the map $\iota_s:\Sigma(\delta_1,\delta_2)\ra \Pi(s)$ extends naturally to an injective map $\iota_s: \Sigma(s)\rightarrow
\Pi(s)$ which is continuous and $\GL_2(\BQ_p)$-equivariant.
\end{prop}
\begin{proof}
If $\delta_s=x^{k-1}$, since the subrepresentation $(\delta_2\circ\det)\otimes\Sym^{k-1}L^2$, which consists of polynomials of degree $\leq k-1$, is contained in $M(s)$, $\iota_s$ induces a map $\Sigma(\delta_1,\delta_2)\ra\Pi(s)$. The injectivity of $\iota_s$ on $\Sigma(\delta_1,\delta_2)$ is proved by Emerton in \cite[Lemma 6.7.2]{Emerton-lg}.
We rewrite his proof in our set up as below for the reader's convenience.

 We first have that
$\iota_s(\Sigma(\delta_1,\delta_2))$ is dense in $\Pi(s)$ since $\LA(\RP^1(\delta_s))$ is dense in $\SC^{u(s)}(\RP^1(\delta_s))$. Hence $\iota_s$ is nonzero because $\Pi(s)\neq 0$ by Theorem \ref{thm:Amice-rank2}. If $w(\delta_s)\notin\mathbb{N}$, $\Sigma(\delta_1,\delta_2)$ is topologically irreducible by Proposition \ref{prop:ST-an}. Thus $\iota_s$ is either injective or zero. It therefore follows that $\iota_s$ must be injective.

In case $w(\delta_s)\in\mathbb{N}$, we put $k=w(s)=w(\delta_s)+1$. We see from Proposition \ref{prop:ST-an} that all the proper admissible subrepresentations of $\Sigma(\delta_1,\delta_2)$ are contained in the image of $\ker(I_k)$. Thus it reduces to show that $\iota_s$ is injective on the image of $\ker(I_k)$.
Note that $k=w(s)>u(s)$. Therefore $\LP^{[0,k-1]}(\BZ_p,L)$, the space of locally polynomial functions of degree $\leq k-1$ on $\BZ_p$, is dense in $\SC^{u(s)}(\BZ_p,L)$ by the classical theorem of Amice-V\'elu and Vishik. We thus deduce that $\iota_s(\ker(I_k))$ is dense in $\Pi(s)$. It follows that $\iota_s(\ker(I_k))$ is infinite dimensional because $\Pi(s)$ is infinite dimensional. If $\delta_s\neq x^{k-1}$, the only possible nontrivial quotient of $\ker(I_k)$ is the finite dimensional representation $(\delta_2|x|^{-1}\circ\det)\otimes\Sym^{k-1}L^2)$. Hence $\iota_s$ must be injective on
$\ker(I_k)$ in this case.
If $\delta_s=x^{k-1}$, note that the image of $\ker(I_k)$ in $\Sigma(\delta_1,\delta_2)$, which is isomorphic to $(\delta_2\circ\det)\otimes\Sym^{k-1}L^2$, is irreducible. It follows that $\iota_s$ is injective on the image of $\ker(I_k)$ as well.

Now suppose $\delta_s=x^{k-1}$. The extension of $\iota_s$ to $\Sigma(s)$ is actually due to Breuil who identifies $\Pi(s)$ with the universal unitary completion of $\Sigma(s)$ and shows that the natural map $\Sigma(s)\ra\Pi(s)$ is injective as long as $\Pi(s)\neq0$ (\cite[Proposition 4.3.5]{Breuil-DS}, \cite[Corollaire 3.3.4]{Breuil-com}). We briefly recall his construction of the natural map $\Sigma(s)\ra\Pi(s)$ as below. One easily sees that Breuil's map extends $\iota_s$.
For any $0\leq i\leq k-1$, and
$$l_i(x)=\sum_{u\in U} \lambda_u (x-a_u)^{j_u}\log_\SL(x-a_u)
$$
where $U$ is a finite set, $j_u$ are integers between $[\frac{k+1}{2}]$ and $k-1$, $\lambda_u\in L$,
$a_u\in\BQ_p$ such that $\deg(\sum_{u\in
U}\lambda_u(x-a_u)^{j_u}+x^i)<u(s)$, it follows from \cite[Lemme 3.3.2]{Breuil-com} that $l_i(x)+x^i\log_{\SL}(x)1_{D(\infty,n)}\in\SC^{u(s)}(\RP^1(\delta_s))$ for $n\in\mathbb{Z}$. We thus define
\[
\int_{D(\infty,n)}x^i\log_{\SL}(x)\mu(x)=\int_{\RP^{1}(\BQ_p)}(l_i(x)+x^i\log_{\SL}(x)1_{D(\infty,n)})\mu(x)
\]
for any $\mu\in\Pi(s)^*$; this is independent of the choice of $l_i(x)$ because the difference of any such two $l_i's$ lies in $M(s)'$ which is killed by $\mu$. By this way, we extend $\mu$ to an element of $C(\delta_s,\SL)^*$. This yields a continuous $\GL_2(\BQ_p)$-equivariant morphism $\Pi(s)^*\ra \Sigma(s)^*$. Taking dual of this morphism, we get Breuil's map $\Sigma(s)\ra\Pi(s)$.
\end{proof}
We are now in the position to reformulate Emerton's conjecture for non-exceptional $s$. Note that $\iota_s(\Sigma(s))\subset\Pi(s)_{\mathrm{an}}$ since $\Sigma(s)$ is a locally analytic representation.
\begin{conj} \label{conj:Emerton-uni}
For $s\in \SS_\irr$ non-exceptional, the cokernel of the inclusion
$\iota_s: \Sigma(s)\rightarrow \Pi(s)_\an$ is isomorphic to
$\widetilde{\Sigma}(\delta_2,\delta_1)$ as locally analytic $\GL_2(\BQ_p)$-representations. Thus the space of
locally analytic vectors $\Pi(s)_\an$ sits in a short exact sequence
of locally analytic $\GL_2(\BQ_p)$-representations
\begin{equation} \label{eq:exact-Emerton}
0\longrightarrow\Sigma(s)\longrightarrow\Pi(s)_\an \longrightarrow
\widetilde{\Sigma}(\delta_2,\delta_1) \longrightarrow 0.
\end{equation}
\end{conj}

\begin{rem}\label{non-split}
Emerton shows that if the above conjecture is true, then the extension (\ref{eq:exact-Emerton}) must be non-split (\cite{Emerton-lg}).
\end{rem}
In the case when $s\in\SS_*^{\mathrm{cris}}$, there is a more explicit description of $\Pi(s)_{\mathrm{an}}$ which is due to Breuil. Recall that $D(s)$ is isomorphic to $D(s')$ for $s'=(x^{w(s)}\delta_2,x^{-w(s)}\delta_1,\SL)$. We thus obtain a morphism $\Sigma((x^{w(s)}\delta_2,x^{-w(s)}\delta_1)\ra\Pi(s')_{\mathrm{an}}\cong\Pi(s)_{\mathrm{an}}$. On the other hand, if $\alpha,\beta:\BQ_p^\times\ra L^\times$ are smooth characters such that $|\alpha(p)|\leq|\beta(p)|$, there is an intertwining from $\Ind^{\GL_2(\BQ_p)}_{\mathrm{B}(\BQ_p)}(\alpha\otimes\beta|x|^{-1})_{\mathrm{sm}}$
to $\Ind^{\GL_2(\BQ_p)}_{\mathrm{B}(\BQ_p)}(\beta\otimes\alpha|x|^{-1})_{\mathrm{sm}}$, yielding an intertwining from $\Sym^{k-1}L^2\otimes\Ind^{\GL_2(\BQ_p)}_{\mathrm{B}(\BQ_p)}(\alpha\otimes\beta|x|^{-1})_{\mathrm{sm}}$
to $\Sym^{k-1}L^2\otimes\Ind^{\GL_2(\BQ_p)}_{\mathrm{B}(\BQ_p)}(\beta\otimes\alpha|x|^{-1})_{\mathrm{sm}}$. It thus follows that if we set $\Sigma(\delta_1, \delta_2)_{\mathrm{lalg}}$ to be the image of $\ker I_k$ in $\Sigma(\delta_1, \delta_2)$, then there exists an intertwining between
$\Sigma(\delta_1,\delta_2)_{\mathrm{lalg}}$ and
$\Sigma(x^{w(s)}\delta_2,x^{-w(s)}\delta_1)_{\mathrm{lalg}}$ which is always injective (but the direction can be either way).
We therefore get a morphism
\begin{equation} \label{eq:breuil}
\Sigma(\delta_1,\delta_2) \ \widetilde{\oplus} \  \Sigma(x^{w(s)}\delta_2,x^{-w(s)}\delta_1)
\ra\Pi(s)_\an
\end{equation}
where $\widetilde{\oplus}$ denotes
the amalgamated sum of two summands over the intertwining between $\Sigma(\delta_1,\delta_2)_{\mathrm{lalg}}$ and
$\Sigma(x^{w(s)}\delta_2,x^{-w(s)}\delta_1)_{\mathrm{lalg}}$.
\begin{conj}$($\cite[Conjectures 4.4.1, 5.3.7]{BB}$)$\label{conj:breuil}
For $s\in\SS^\cris_*$ non-exceptional,
(\ref{eq:breuil}) is a topological isomorphism.
\end{conj}
\begin{prop}\label{prop:Emerton=Breuil}
For $s\in\SS^\cris_*$ non-exceptional, Emerton's conjecture is equivalent to Breuil's conjecture.
\end{prop}
\begin{proof}
The generic case that $V(s)$ does not admit an $\SL$-invariant (this is equivalent to $\delta_s\neq
x^{w(s)-1}, x^{w(s)-1}|x|^{-2}$) is already proved in
\cite[6.7.5]{Emerton-lg}. We now prove the remaining cases. The injectivity of (\ref{eq:breuil}) is already ensured by \cite[Corollaires 5.3.6, 5.4.3]{BB}. It reduces to show that
$\Sigma(\delta_1,\delta_2) \ \widetilde{\oplus} \  \Sigma(x^{w(s)}\delta_2,x^{-w(s)}\delta_1)$
and $\Sigma(s)\oplus \widetilde{\Sigma}(\delta_2,\delta_1)$ have same constitutes.
If $\delta_s=x^{w(s)-1}|x|^{-2}$, then $\Sigma(s)=\Sigma(\delta_1,\delta_2)$
and the intertwining is $\Sigma(x^{w(s)}\delta_2,x^{-w(s)}\delta_1)_{\mathrm{lalg}}\ra\Sigma(\delta_1,\delta_2)_{\mathrm{lalg}}$.
Therefore
$$\left(\Sigma(\delta_1,\delta_2) \ \widetilde{\oplus} \
\Sigma(x^{w(s)}\delta_2,x^{-w(s)}\delta_1)\right)/\Sigma(s)
=\Sigma(x^{w(s)}\delta_2,x^{-w(s)}\delta_1)/\Sigma(x^{w(s)}\delta_2,x^{-w(s)}\delta_1)_{\mathrm{lalg}}\cong
\widetilde{\Sigma}(\delta_2,\delta_1)$$
by Proposition \ref{prop:ST-an}. If $\delta_s=x^{w(s)-1}$, the intertwining is $\Sigma(\delta_1,\delta_2)_{\mathrm{lalg}}\ra\Sigma(x^{w(s)}\delta_2,x^{-w(s)}\delta_1)_{\mathrm{lalg}}$ and the quotient is isomorphic to $(\delta_2\circ\det)\otimes\Sym^{w(s)-1}L^2$. Thus
$$\left(\Sigma(\delta_1,\delta_2) \ \widetilde{\oplus} \
\Sigma(x^{w(s)}\delta_2,x^{-w(s)}\delta_1)\right)/\Sigma(\delta_1,\delta_2)$$
is an extension of $\Sigma(x^{w(s)}\delta_2,x^{-w(s)}\delta_1)/\Sigma(x^{w(s)}\delta_2,x^{-w(s)}\delta_1)_{\mathrm{lalg}}
\cong\widetilde{\Sigma}(\delta_2,\delta_1)$ by
$(\delta_2\otimes\det)\otimes \Sym^{w(s)-1}L^2$. We thus obtain that $\Sigma(\delta_1,\delta_2) \ \widetilde{\oplus} \  \Sigma(x^{w(s)}\delta_2,x^{-w(s)}\delta_1)$ has same constitutes with $\Sigma(s)\oplus \widetilde{\Sigma}(\delta_2,\delta_1)$ in both cases.
\end{proof}

\subsection{An exact sequence}

From now on, we suppose $p>2$. Recall that for any $s=(\delta_1,\delta_2,\SL)\in \SS_\irr$, there is an exact sequence
\[
0\longrightarrow\SR(\delta_1)\stackrel{i}{\longrightarrow} D(s)\stackrel{j}{\longrightarrow}\SR(\delta_2)\longrightarrow0.
\]
For $i=1,2$,
we denote $\mathscr{A}_{\delta_s,\delta_i}, \SR^+(\delta_i)\boxtimes_{\delta_s}\RP^1, \SR(\delta_i)\boxtimes_{\delta_s}\RP^1$ by $\mathscr{A}_i,\SR^+(\delta_i)\boxtimes\RP^1,\SR(\delta_i)\boxtimes\RP^1$ for simplicity. Recall that $\mathscr{A}_s:\Pi(\check{s})^*\ra D^\natural(s)\boxtimes\RP^1$ is the topological $\GL_2(\BQ_p)$-equivariant isomorphism given by Theorem \ref{thm:Amice-rank2}.

\begin{prop}\label{prop:exact-sq-left}
If $s$ is non-exceptional, then the $\GL_2(\BQ_p)$-equivariant morphism
\[
\mathscr{A}_2\circ\iota_{\check{s}}^*\circ\mathscr{A}_s^{-1}:D^\natural(s)\boxtimes \RP^1
\ra \SR^+(\delta_2) \boxtimes\RP^1
\]
satisfies $\mathscr{A}_2\circ\iota_{\check{s}}^*\circ\mathscr{A}_s^{-1}((z,z'))=(j(z),j(z'))$ for any $(z,z')\in D^\natural(s)\boxtimes \RP^1$.
\end{prop}

\begin{proof}
Suppose the images of $z,z'$ in $\SR(\delta_2)$ are $z_2e_2,z_2'e_2$ where $e_2$ is the standard basis of $\SR(\delta_2)$ fixed in $\S4.2$. Suppose $\mathscr{A}_s^{-1}((z,z'))=\mu$. Then
by the definition of $\mathscr{A}_s$, we see that $\mathscr{A}(\mu|_{\BZ_p})=z_2, \mathscr{A}(w\mu|_{\BZ_p})=z'_2$. Hence
$\mathscr{A}_2\circ\iota_s^*\circ\mathscr{A}_s^{-1}((z,z'))=
(z_2\otimes e_2,z_2'\otimes e_2)=(j(z),j(z'))$.
\end{proof}

\begin{cor} \label{prop:exact-sq-right}
If $s$ is non-exceptional, then $j(w_D(x))=w_{\delta_s}(j(x))$ for any $x\in D(s)\boxtimes\BZ_p^\times$.
\end{cor}

\begin{proof}
If $x\in (1-\varphi)D(s)^{\psi=1}$, by the proof of \cite[Proposition V.2.1]{Colmez-Langlands}, there exists $z=(z_1,z_2)\in D^\natural(s)\boxtimes \RP^1$ such that $\Res_{\BZ_p^\times}z_1=x$.
Since $j$ commutes with $\Res_{\BZ_p^\times}$ and $(j(z_1),j(z_2))\in \SR^+(\delta_2) \boxtimes\RP^1$, we get
\[
w_{\delta_s}(j(x))=w_{\delta_s}(j(\Res_{\BZ_p^\times}(z_1)))
=w_{\delta_s}(\Res_{\BZ_p^\times}(j(z_1)))=\Res_{\BZ_p^\times}(j(z_2))=j(\Res_{\BZ_p^\times}(z_2))=j(w_D(x)).
\]
By \cite[Corollaire V.1.13]{Colmez-Langlands}, $D(s)\boxtimes\BZ_p^\times$ is generated by $(1-\varphi)D(s)^{\psi=1}$ as an $\SR(\Gamma)$-module. We conclude the corollary by the case $x\in (1-\varphi)D(s)^{\psi=1}$ and the fact that both $w_D,w_{\delta_s}$ are $\SR(\Gamma)$-antilinear.
\end{proof}

\begin{prop} \label{prop:exact-sq}
The maps
$$i_{\RP^1}: \SR(\delta_1)\boxtimes \RP^1 \rightarrow
D(s)\boxtimes \RP^1, \hskip 5pt (z_1,z_2)\mapsto
(i(z_1), i(z_2))$$ and $$j_{\RP^1}: D(s)\boxtimes \RP^1
\rightarrow \SR(\delta_2)\boxtimes\RP^1, \hskip 5pt
(z_1,z_2)\mapsto (j(z_1), j(z_2))$$ are well-defined morphisms of continuous $\GL_2(\BQ_p)$-representations. Moreover, we have the short exact sequence
\begin{equation} \label{eq:exact-sq-proof}
 0 \longrightarrow \SR(\delta_1)\boxtimes\RP^1
\stackrel{i_{\RP^1}}{\longrightarrow} D(s) \boxtimes \RP^1 \stackrel{j_{\RP^1}}{\longrightarrow}\SR(\delta_2)\boxtimes\RP^1
\longrightarrow 0.
\end{equation}
\end{prop}


\begin{proof}
By Propositions \ref{prop:GL2action}, \ref{prop:rank1-GL2action}, the $\GL_2(\BQ_p)$-actions on  $\SR(\delta_1)\boxtimes\RP^1,\SR(\delta_2)\boxtimes\RP^1$ and $D(s)\boxtimes \RP^1$ satisfy the same set of formulas. It thus follows that $i_{\RP^1}, j_{\RP^1}$ are
$\GL_2(\BQ_p)$-equivariant as long as they are well-defined. In general, the formula (\ref{eq:wD}) of $w_D$ is not convergent on $D^\dagger\boxtimes\BZ_p^\times$. Hence the well-defineness of $i_{\RP^1}$ and $j_{\RP^1}$ are not obvious from their definitions.

The well-defineness of $j_{\RP^1}$ follows from Corollary \ref{prop:exact-sq-right}. To show that $i_{\RP^1}$ is well-defined, we use the pairing $\{\cdot,\cdot\}$. First note that $i:\SR(\delta_1)\ra D(s)$ is dual to $j:D(\check{s})\ra\SR(\check{\delta}_1)$ with respect to $\langle\cdot,\cdot\rangle$. It therefore follows that $i:\SR(\delta_1)\boxtimes\BZ_p^\times\ra D(s)\boxtimes\BZ_p^\times$ is the dual of $j:D(\check{s})\boxtimes\BZ_p^\times\ra\SR(\check{\delta}_1)\boxtimes\BZ_p^\times$ with respect to $\langle\cdot,\cdot\rangle$.  Hence $i:\SR(\delta_1)\boxtimes\BZ_p^\times\ra D(s)\boxtimes\BZ_p^\times$ is the dual of $j:D(\check{s})\boxtimes\BZ_p^\times\ra\SR(\check{\delta}_1)\boxtimes\BZ_p^\times$ with respect to $\{\cdot,\cdot\}$. Since $\{\cdot,\cdot\}$ is $w$-invariant and we have already proved that $j$ commutes with $w$ on $D(\check{s})\boxtimes\BZ_p^\times$, we thus deduce that $i$ commutes with $w$ on $\SR(\delta_1)\boxtimes\BZ_p^\times$. Hence $i_{\RP^1}:\SR(\delta_1)\boxtimes \RP^1 \rightarrow
D(s)\boxtimes \RP^1$ is well-defined.

For $(\ref{eq:exact-sq-proof})$, the injectivity of $i_{\RP^1}$ and the exactness at $D(s)\boxtimes\RP^1$ is obvious. To show the surjectivity of $j_{\RP^1}$, for any $(z,z')\in \SR(\delta_2)\boxtimes\RP^1$, we pick $y\in D(s)$, $y'\in D(s)\boxtimes p\BZ_p$ which lift $z, \Res_{p\BZ_p}z'$ respectively. Then an easy computation shows that $j_{\RP^1}(y,y'+w_{D}(\Res_{\BZ_p^\times}y))=(z,z')$. This proves that $j_{\RP^1}$ is surjective.
\end{proof}

Henceforthe we identify $\SR(\delta_1)\boxtimes\RP^1$ with a submodule of $D(s)\boxtimes\RP^1$ via $i_{\RP^1}$.

\subsection{Proof of Emerton's conjecture}
We prove Conjecture
\ref{conj:Emerton-uni} for $p>2$ in this subsection. From now on, let $s\in \SS_\irr$ be
non-exceptional. Since $\mathscr{A}_s$ is a topological $\GL_2(\BQ_p)$-equivariant isomorphism between the contragredient representation $\Pi(\check{s})^*$ and $D^\natural(\check{s})\boxtimes \RP^1$, it induces an isomorphism
\[
\mathscr{A}_{s,\mathrm{an}}:
(\Pi(\check{s})_\an)^* \rightarrow ((D^\natural(s)\boxtimes\RP^1)_{\mathrm{an}})^*=D^\natural_\rig(s) \boxtimes \RP^1
\]
between coadmissible $D(\GL_2(\BZ_p))$-modules $(\Pi(\check{s})_\an)^*$ and $D^\natural_\rig(s) \boxtimes \RP^1$, where $D(\GL_2(\BZ_p))$ denotes the algebra of locally analytic distributions on $\GL_2(\BZ_p)$.

\begin{prop} \label{prop-comm-diag}
The diagram
\begin{equation} \label{eq:comm-diag}
\xymatrix{
(\Pi(\check{s})_\an)^* \ar[r]^{\mathscr{A}_{s,\an}} \ar[d]^{\iota_{\check{s}}^*} &
D^\natural_\rig(s) \boxtimes \RP^1 \ar[d]_{j_{\RP^1}} \\
\widetilde{\Sigma}(\check{\delta_2}, \check{\delta_1})^*
\ar[r]^{\mathscr{A}_2} & \SR(\delta_2)\boxtimes\RP^1 }
\end{equation}
is commutative. As a consequence, we have $j_{\RP^1}(D^\natural_\rig(s) \boxtimes \RP^1)=\mathscr{A}_{2}(\Sigma(\check{\delta_2}, \check{\delta_1})^*)$
\end{prop}

\begin{proof}
Recall that for an admissible Banach space representation $U$ of $\GL_2(\BQ_p)$, $U_{\mathrm{an}}$ is dense in $U$ (\cite[Theorem 7.1]{ST-an3}); hence $U^*$ is dense in $U_{\mathrm{an}}^*$.  The diagram (\ref{eq:comm-diag}) commutes on $\Pi(\check{s})^*\subset(\Pi(\check{s})_\an)^*$ following Proposition \ref{prop:exact-sq-left}. We thus conclude the commutativity of (\ref{eq:comm-diag}) by the density of $\Pi(\check{s})^*$ in $(\Pi(\check{s})_\an)^*$. It thus follows that $j_{\RP^1}(D^\natural_\rig(s) \boxtimes \RP^1)=\mathscr{A}_{2}(\iota_{\check{s}}^*((\Pi(\check{s})_\an)^*))=\mathscr{A}_{2}(\Sigma(\check{\delta_2}, \check{\delta_1})^*).$
\end{proof}

\begin{lem}\label{lem:SR+orthogonal}
$\SR^+(\check{\eta})\boxtimes_{\delta^{-1}}\RP^1$ and $\SR^+(\eta)\boxtimes_{\delta}\RP^1$ are orthogonal complements of each other under the pairing $\SR(\check{\eta})\boxtimes_{\delta^{-1}}\RP^1\times
\SR(\eta)\boxtimes_{\delta}\RP^1\rightarrow L$.
\end{lem}
\begin{proof}
It suffices to show that $\SR^+$ is the orthogonal complements of itself under the pairing $\{\cdot,\cdot\}:\SR\times\SR\ra L$. It is obvious that $\{\SR^+,\SR^+\}=0$. On the other hand, if $f=\sum_{i\in\BZ}a_iT^i\in (\SR^+)^{\bot}$, then for any $j\in\BN$, $\{\sigma_{-1}(T^j),f\}=a_{-j-1}$ implies $a_{-j-1}=0$, yielding $f\in\SR^+$.
\end{proof}

\begin{lem} \label{lem:orthogonal}
$j_{\RP^1}( D^\natural_\rig(s)
\boxtimes \RP^1)\subset\SR(\delta_2) \boxtimes \RP^1$
and $D^\natural_\rig(\check{s}) \boxtimes \RP^1 \cap
\SR(\check{\delta}_2) \boxtimes \RP^1$ are
orthogonal complements of each other under $\{ \cdot ,   \cdot\}_{\RP^1}:\SR(\delta_2)\boxtimes\RP^1\times\SR(\check{\delta}_2) \boxtimes \RP^1\ra L$.
\end{lem}

\begin{proof}
By the constructions of $i_{\RP^1}$ and $j_{\RP^1}$, one easily checks that
$i_{\RP^1}: \SR(\check{\delta_2})\boxtimes \RP^1\rightarrow
D(\check{s})\boxtimes \RP^1$ is dual to
$j_{\RP^1}: D(s)\boxtimes\RP^1\rightarrow \SR(\delta_2)\boxtimes
\RP^1$ with respect to $\{\cdot,\cdot\}_{\RP^1}$. Thus by Proposition \ref{prop:Drig-orthogonal}, we deduce
that
\[
\{j_{\RP^1}(x),y \}_{\RP^1}=\{x,y\}_{\RP^1}=0
\]
for any $x \in D^\natural_\rig(s) \boxtimes \RP^1$ and $y\in D^\natural_\rig(\check{s}) \boxtimes \RP^1 \cap
\SR(\check{\delta_2}) \boxtimes\RP^1$. This proves $j_{\RP^1}(D^\natural_\rig(s) \boxtimes \RP^1)\subseteq (D^\natural_\rig(\check{s}) \boxtimes \RP^1 \cap
\SR(\check{\delta_2}) \boxtimes\RP^1)^{\bot}$.

On the other hand, since $\Sigma(\check{\delta_1}, \check{\delta_2})$ and $\widetilde{\Sigma}(\check{\delta_1}, \check{\delta_2})$ are admissible locally analytic representations, $\Sigma(\check{\delta_1}, \check{\delta_2})^*$ and $\widetilde{\Sigma}(\check{\delta_1}, \check{\delta_2})^*$ are coadmissible $D(\GL_2(\BZ_p))$-modules. Therefore $\Sigma(\check{\delta_1}, \check{\delta_2})^*$ is a closed subspace of $\widetilde{\Sigma}(\check{\delta_1}, \check{\delta_2})^*$ by \cite[Lemma 3.6]{ST-an3}. This implies that $j_{\RP^1}(D^\natural_\rig(\check{s}) \boxtimes \RP^1)=\mathscr{A}_2(\Sigma(\check{\delta_1}, \check{\delta_2})^*)$ is a closed subspace of $\SR^+(\check{\delta}_1)\boxtimes\RP^1$ by Proposition \ref{prop-comm-diag}; hence $j_{\RP^1}(D^\natural_\rig(\check{s}) \boxtimes \RP^1)$ is Fr\'echet complete with the subspace topology of $\SR(\check{\delta}_1)\boxtimes\RP^1$. By the open mapping theorem for Fr\'echet type spaces (\cite[Proposition 8.8]{Schneider}), we deduce that $j_{\RP^1}:D^\natural_\rig(\check{s}) \boxtimes \RP^1\ra j_{\RP^1}(D^\natural_\rig(\check{s}) \boxtimes \RP^1)$ is open. Therefore the quotient topology and the subspace topology on $j_{\RP^1}(D^\natural_\rig(\check{s}) \boxtimes \RP^1)$ coincide.

Now for any $x\in(D^\natural_\rig(\check{s}) \boxtimes \RP^1 \cap
\SR(\check{\delta_2}) \boxtimes\RP^1)^{\bot}\subset\SR(\delta_2)\boxtimes\RP^1$,  we pick $\tilde{x}\in D(s) \boxtimes
\RP^1$ lifting $x$. The continuous linear functional $f(y)=\{\tilde{x},y
\}_{\RP^1}$ on $D_\rig^\natural(\check{s}) \boxtimes \RP^1$
induces a continuous linear functional $\bar{f}$ on $j_{\RP^1}(D_\rig^\natural(\check{s}) \boxtimes \RP^1)$. Applying Hahn-Banach theorem for Fr\'echet type spaces (\cite[Corollary 9.4]{Schneider}), we extend $\bar{f}$ to a continuous linear functional on
$\SR(\check{\delta_1})\boxtimes\RP^1$. Since the pairing
$\SR(\check{\delta_1})\boxtimes\RP^1\times\SR(\delta_1)\boxtimes\RP^1\ra L$ is perfect, we may suppose that the extension of $\bar{f}$ is defined by some $x'\in \SR(\delta_1)\boxtimes
\RP^1$. It therefore follows that for any $y\in D^\natural_\rig(\check{s}) \boxtimes \RP^1  $,
\[
\{\tilde{x}-x', y\}_{\RP^1} = \{\tilde{x},y\}_{\RP^1}-\{x',j_{\RP^1}(y)\}_{\RP^1}=\bar{f}(j_{\RP^1}(y))-\{x',j_{\RP^1}(y)\}_{\RP^1}=0,
\]
 yielding $\tilde{x}-x'\in D^\natural_\rig(s) \boxtimes
\RP^1$. We thus conclude that
$x\in j_{\RP^1}(D^\natural_\rig(s) \boxtimes \RP^1)$ because
$j_{\RP^1}(\tilde{x}-x')=x$. This proves $(D^\natural_\rig(\check{s}) \boxtimes \RP^1 \cap
\SR(\check{\delta_2}) \boxtimes\RP^1)^{\bot}\subseteq j_{\RP^1}(D^\natural_\rig(s) \boxtimes \RP^1)$.
\end{proof}

\begin{prop} \label{prop:analytic-vector-0} The following are true:
\begin{enumerate}
\item \label{it:has-L} if $\delta_s=x^{k-1}$, where $k$ is an integer $\geq 2$,
then $\SR(\delta_1)\boxtimes \RP^1 \cap
D^\natural_\rig(s) \boxtimes \RP^1$ contains
$\SR^+(\delta_1)\boxtimes\RP^1$ as a closed subspace of
codimension  $k$, and $j_{\RP^1}(D^\natural_\rig(s) \boxtimes
\RP^1)$ is a closed subspace of
$\SR^+(\delta_2)\boxtimes\RP^1$ of codimension $k$;
\item \label{it:no-L} otherwise, $\SR(\delta_1)\boxtimes \RP^1 \cap
D^\natural_\rig(s) \boxtimes \RP^1=\SR^+(\delta_1)\boxtimes
\RP^1$ and $j_{\RP^1}(D^\natural_\rig(s) \boxtimes \RP^1)
=\SR^+(\delta_2)\boxtimes \RP^1$.
\end{enumerate}
\end{prop}

\begin{proof}
We prove $(\mathrm{i})$ only. The proof of $(\mathrm{ii})$ is similar. For $(\mathrm{i})$, it follows from Proposition \ref{prop-comm-diag} that $j_{\RP^1}(D^\natural_\rig(s) \boxtimes \RP^1)=\mathscr{A}_2(\Sigma(\check{\delta_2}, \check{\delta_1})^*)$. Recall that $\Sigma(\check{\delta_2}, \check{\delta_1})$ is a quotient of $\widetilde{\Sigma}(\check{\delta_2},\check{\delta_1})$
 by a $k$-dimensional subrepresentation. Hence $\Sigma(\check{\delta_2}, \check{\delta_1})^*$ is a closed subspace of
$\widetilde{\Sigma}(\check{\delta_2}, \check{\delta_1})^*$ of codimension $k$, yielding that $j_{\RP^1}(D^\natural_\rig(s) \boxtimes
\RP^1)$ is a closed subspace of
$\SR^+(\delta_2)\boxtimes\RP^1$ of codimension $k$.

On the other hand, as
\[
 \SR(\delta_1)\boxtimes \RP^1 \cap
D^\natural_\rig(s) \boxtimes \RP^1 =j_{\RP^1}(D^\natural_\rig(\check{s}) \boxtimes \RP^1)^{\bot} \quad\text{and}\quad
\SR^+(\delta_1)\boxtimes\RP^1=(\SR^+(\check{\delta}_1) \boxtimes \RP^1)^{\bot}
\]
by Lemmas \ref{lem:orthogonal}, \ref{lem:SR+orthogonal}, we deduce that
$\SR^+(\delta_1)\boxtimes\RP^1$ is a codimension $k$ closed subspace of
 $\SR(\delta_1)\boxtimes \RP^1 \cap
D^\natural_\rig(s)$.
\end{proof}

\begin{thm}\label{thm:emerton}
Conjecture \ref{conj:Emerton-uni} is true for $p>2$.
\end{thm}

\begin{proof}
By Proposition \ref{prop:analytic-vector-0},
$\SR^+(\check{\delta_2})\boxtimes\RP^1$
is contained in
$D^\natural_\rig(\check{s})\boxtimes \RP^1$. Let
$\Sigma$ be the locally analytic representation such that $\Sigma^*$ is isomorphic to
$D^\natural_\rig(\check{s}) \boxtimes \RP^1/
\SR^+(\check{\delta_2})\boxtimes_{\check{\delta}}\RP^1$. Since
$\SR^+(\check{\delta_2})\boxtimes\RP^1$ is isomorphic to $\widetilde{\Sigma}(\delta_2,\delta_1)^*$, we thus have
an exact sequence of locally analytic $L$-representations of
$\GL_2(\BQ_p)$
\begin{equation} \label{equation:pre-conj}
0\longrightarrow \Sigma \stackrel{\iota_1}{\longrightarrow}\Pi(s)_\an \stackrel{\iota_2}{\longrightarrow}
\widetilde{\Sigma}(\delta_2,\delta_1)\longrightarrow 0.
\end{equation}
If $\delta_s$ is not of the form $x^{k-1}$ for any $k\in\BZ_+$,
then $\Sigma^*\cong\SR^+(\check{\delta_1})\boxtimes\RP^1$ by Proposition \ref{prop:analytic-vector-0}(\ref{it:no-L}), which in turn is isomorphic to the dual of $\widetilde{\Sigma}(\delta_1,\delta_2)$.
We thus have that $\Sigma$ is isomorphic to $\widetilde{\Sigma}(\delta_1,\delta_2)=\Sigma(s)$, yielding $(\ref{eq:exact-Emerton})$ in this case. Now suppose $\delta_s=x^{k-1}$ for some integer $k\geq 1$.
Since $\widetilde{\Sigma}(\delta_2,\delta_1)$ is
topologically irreducible, and it is not isomorphic to any topological irreducible subquotients of $\Sigma(s)$ by Proposition \ref{prop:ST-an}, we deduce that $\iota_2(\iota_s(\Sigma(s)))=0$. Hence $\iota_s(\Sigma(s))\subseteq\iota_1(\Sigma)$.
On the other hand, by Proposition \ref{prop:analytic-vector-0} (\ref{it:has-L}), we see
that $\Sigma^*$ is an extension of $\Sigma(\delta_1, \delta_2)^*$ by
a $k$-dimensional $L$-vector space. Hence $\Sigma$ contains
$\Sigma(\delta_1,\delta_2)$ as a subrepresentation of codimension $k$. Since $\Sigma(\delta_1,\delta_2)$ is a subrepresentation of $\Sigma(s)$ of codimension $k$ as well, we conclude that $\iota_s(\Sigma(s)) = \iota_1(\Sigma)$.
\end{proof}
\begin{rem}\label{rem:JH-dual}
As a consequence of Theorem \ref{thm:emerton} and Proposition \ref{prop:analytic-vector-0}, we see that in case $\delta_s=x^{k-1}$ for some $k\in\BZ_+$, the dual of the quotient $\Pi(s)_\an/\Sigma(\delta_1,\delta_2)$, which is an extension of $\widetilde{\Sigma}(\delta_2,\delta_1)$ by $(\delta_2\circ\det)\otimes\Sym^{k-1}L^2$, is isomorphic to $\SR(\check{\delta}_2)\boxtimes\RP^1\cap D^\natural_\rig(\check{s})\boxtimes\RP^1$.
\end{rem}

\end{document}